\documentclass[english, 11pt, a4paper]{article}
\PassOptionsToPackage{dvipsnames}{xcolor}
\usepackage{babel}
\usepackage[utf8]{inputenc}
\usepackage[T1]{fontenc}
\usepackage[top=2.5cm, bottom=2.5cm, left=2.5cm, right=2.5cm]{geometry} 
\usepackage{verbatim}
\usepackage{amsmath}
\usepackage{amsthm}
\usepackage{amsfonts}
\usepackage{mathabx}
\usepackage{color}
\usepackage{xcolor}
\usepackage{amssymb}
\usepackage{mathrsfs}
\usepackage{subfigure}
\usepackage{float}
\usepackage[most]{tcolorbox}
\usepackage[breaklinks=true]{hyperref}
\usepackage{enumitem}
\usepackage{bbm}
\usepackage{xspace}
\usepackage{regexpatch}
\usepackage{colortbl}%
\usepackage{enumitem}
\usepackage{makecell}
\usepackage{microtype}
\usepackage{mathtools,algorithm}
\usepackage[toc,page]{appendix}
\usepackage[noend]{algpseudocode}
\usepackage{subcaption}

\DeclareMathAlphabet{\mathb}{OML}{cmm}{b}{it}

\newcommand{\N}{\mathbb{N}}		
\newcommand{\R}{\mathbb{R}}		

\newcommand{\fonc}[3]{#1:  #2  \rightarrow  #3}					
\newcommand{\syst}[1]{\left \{ \begin{array}{l} #1 \end{array} \right. \kern-\nulldelimiterspace}	
\newcommand{\prox}{\mathrm{prox}}
\newcommand{\argmin}{\mathrm{argmin}}

\newcommand{\dom}{\mathrm{dom}\,}

\renewcommand{\int}{\mathrm{int}\,}

\newcommand{\minimize}[2]{\ensuremath{\underset{\substack{{#1}}}{\mathrm{minimize}}\;\;#2 }}

\newcommand{\argmind}[2]{\ensuremath{\underset{\substack{{#1}}}%
{\mathrm{argmin}}\;\;#2 }}

\newcommand{\lsc}{\Phi_{lsc}}
\newcommand{\proxlsc}{\mathrm{prox}^{\mathrm{lsc}}}

\makeatletter
\newlength{\algorithmboxrule}
\newcommand{\algorithmboxcolor}{white}
\xpatchcmd*{\algocf@caption@boxruled}{0.0pt}{2\algorithmboxrule}{}{}
\xpatchcmd*{\algocf@caption@boxruled}{\vrule}{\vrule width \algorithmboxrule}{}{}
\xpatchcmd{\algocf@caption@boxruled}{\hrule}{\hrule height \algorithmboxrule}{}{}
\xpretocmd{\algocf@caption@boxruled}{\color{\algorithmboxcolor}}{}{}
\makeatother

\newtheorem{lemma}{Lemma}[section]
\newtheorem{proposition}[lemma]{Proposition}

\newtheorem{theorem}[lemma]{Theorem}
\theoremstyle{definition}
\newtheorem{definition}[lemma]{Definition}
\newtheorem{remark}[lemma]{Remark}
\newtheorem{example}[lemma]{Example}

\newtheorem{assumption}[lemma]{Assumption}

\title{Primal-Dual algorithms for Abstract convex functions with respect to quadratic functions}
\author{ Ewa Bednarczuk\thanks{Warsaw University of Technology,  00-662 Warsaw, Koszykowa 75, Poland} \thanks{Systems Research Institute, PAS,  01-447 Warsaw, Newelska 6, Poland}
\and
The Hung Tran\footnotemark[2]}

\begin{document}
\maketitle

\begin{abstract}
We consider the saddle point problem where the objective functions are abstract convex with respect to the class of quadratic functions. We propose primal-dual algorithms using the corresponding abstract proximal operator and investigate the convergence under certain restrictions. We test our algorithms by several numerical examples. 
\end{abstract}

\section{Introduction}
We consider an instance of the following saddle point problem
\begin{equation}
    \label{prob:DCP}
\inf_{x\in X} \sup_{\psi \in \Psi} f(x)+\psi (Lx) -g^*_\Psi (\psi),  \tag{SP}
\end{equation}
where $X$ and $Y$ are Hilbert spaces, $f:X\to(-\infty,+\infty],g:Y\to(-\infty,+\infty]$ are proper $\Phi$-convex and $\Psi$-convex respectively with respect to the class of functions $\Phi$ and $\Psi$, while $L:X\to Y$ is a continuous linear operator. 
{ For the problem to be well-defined, we assume that $L(\dom f)\cap \dom g\neq \emptyset$, where $\dom f=\{x\in X: f(x)<+\infty\}$ is the domain of $f$.}
The $\Phi$-convexity is a generalization of convexity which is called abstract convex. The concept of abstract convex functions had been formally introduced by Rubinov \cite{Rub2013}, Pallaschke and Rolewicz \cite{Pall2013}, Singer \cite{sin1997} i.e. given a collection of function $\Psi:=\{\psi:X\to \mathbb{R}\}$, a function $f:X\to (-\infty,+\infty]$ is called $\Psi$-convex if it can be written as
\[
(\forall x\in X) \quad f(x) = \sup_{\psi\in\Psi,\psi\leq f} \psi (x).
\]

\emph{This paper aims to investigate the convergence of primal-dual algorithm for problem \eqref{prob:DCP} when $g$ is convex and $g$ is $\Phi$-convex with $L$ is an identity operator.}

It has been known that a proper lsc convex function is abstract convex with respect to the class of affine functions \cite[Corollary 13.42]{Bau2011}. 
For a more general functions of $\Phi$, the class of $\Phi$-convex encompass both convex and nonconvex functions e.g. star-convex \cite{rubinov1999}, weakly/strongly convex \cite{bednarczuk2023duality}.
Recent attention has been focused on algorithms development for nonconvex problems and the topic of abstract convexity has gained popularity \cite{laude2023anisotropic,rakotomandimby2024subgradient,laude2025anisotropic,oikonomidis2025forward}. 
Because with abstract convexity, one can define additional structures similar to the classical convex function such as conjugation and subdifferentials. Therefore, abstract convex functions can be analyzed globally. This aspect allows us to design algorithms which converges to the global minima in the nonconvex setting, while standard analysis only gives convergence to a critical point. On the other hand, to obtain convergence in nonconvex setting, additional conditions are needed, such as K{\L} (or P{\L}) inequality \cite{Bolte2014_Proximal}, error bound condition \cite{cuong2022error}, subdifferentials error bound \cite{atenas2023unified} etc. In practice, these conditions can be difficult to verify as the knowledge of critical points are needed.  

A typical case of problem \eqref{prob:DCP} is the saddle point problem 
\begin{equation}
    \label{prob: dual g convex, f weakly convex}
    \inf_{x\in X} \sup_{y\in Y} f(x)+\langle Lx,y\rangle -g^*(y),
\end{equation}
where $g:Y\to (\infty,+\infty]$ is proper lsc convex and $L:X\to Y$ is a bounded linear operator. Saddle point problem, or min-max optimization has been gaining popularity in applications, especially in machine learning, including generative adversarial networks (GANs) \cite{goodfellow2014generative}, deep learning \cite{sinha2018certifying}, distributed computing \cite{mateos2010distributed} and much more.

Our contribution in this paper is the discussion of primal-dual algorithm for solving problem \eqref{prob: dual g convex, f weakly convex} when $f$ is $\Phi$-convex. This is an extension of our paper \cite{bednarczuk2025proximal} using the $\Phi$-proximal operator for solving nonconvex problem. We obtain convergence of proximal point algorithm, forward-backward algorithm for $\Phi$-convex function without additional regularity condition. Hence, we attempt to investigate the convergence of primal-dual algorithm for the saddle point problem in the presence of $\Phi$-convexity. 

The structure of the paper is as follow: In Section \ref{sec: preli}, we present the definition of abstract convex function as well as its properties when $\Phi=\Phi_{lsc}^\mathbb{R}$ is the class of quadratic function. We introduce the defintion of $\Phi_{lsc}^\mathbb{R}$-proximal operator and its characteristics. In Section \ref{sec: pd alg lsc prox}, we propose $\Phi_{lsc}^\mathbb{R}$ primal-dual algorithm for solving problem \eqref{prob: dual g convex, f weakly convex} and discuss its convergence in Theorem \ref{thm: pd lsc prox convergence}. In Section \ref{sec: full lsc problem}, we examine our primal-dual algorithm for problem \eqref{prob:DCP} in the full $\Phi_{lsc}^\mathbb{R}$-setting. Finally, we present some numerical examples demonstrating our discussed algorithms in Section \ref{sec: lsc-pd numerical}.

\section{Preliminaries}
\label{sec: preli}
\subsection{\texorpdfstring{$\Phi_{lsc}^\mathbb{R}$}-Convex Functions}
In this paper, we consider $X$ to be Hilbert with inner product $\langle \cdot,\cdot\rangle: X\times X\to \mathbb{R}$ and the norm $\Vert x\Vert =\sqrt{\langle x,x\rangle}$.
For two functions $f,g:X\to (-\infty,+\infty]$, $f\leq g$ means that $f(x)\leq g(x)$ for all $x\in X$.
The domain of $f$ is the set $\text{dom }f=\{ x\in X: f(x)<+\infty\}$, and $f$ is proper if $\dom f \neq \emptyset$ and $f(x)>-\infty$ for every $x\in X$. We say that $f$ is lower semi-continuous (lsc) at $x\in X$ if $f(x)\leq \liminf_{x_n \to x} f(x_n)$.
We denote $\mathbb{R}_+, \mathbb{R}_-$ as the set of nonnegative and non-positive number, respectively.
We use the following convention: $+\infty - \infty = +\infty$.
Let us define the class of quadratic function of the form
\begin{equation}
\label{def: lsc class}
    \Phi_{lsc}^\mathbb{R} :=\{ \phi=(a,u)\in \mathbb{R}\times X: \phi(x) = -a \Vert x\Vert^2 +\langle u,x\rangle\}.
\end{equation}
\begin{definition}
Let $X$ be a Hilbert space and $f:X\to(-\infty,+\infty]$. We say that $f$ is $\Phi_{lsc}^\mathbb{R}$-convex on $X$ if the following holds,
\begin{equation*}
(\forall x\in X) \quad f(x) = \sup_{\phi\in\Phi_{lsc}^\mathbb{R}, \phi \leq f} \phi (x).    
\end{equation*}
\end{definition}
The definition of $\Phi_{\lsc}^\mathbb{R}$-convexity belongs a general concept called abstract convexity which has been investigated in \cite{Rub2013,Pall2013,sin1997} to develop convexity without linearity. 
The reason we choose the class $\Phi_{lsc}^\mathbb{R}$ is that it covers a large class of functions with nice properties and it is closely related to strongly and weakly convex functions \cite{bednarczuk2025proximal}. It has been proved in \cite[Proposition 6.3]{Rub2013} that all lower-semicontinous functions is $\Phi_{lsc}^\mathbb{R}$-convex. When $\Phi_{lsc}^\mathbb{R}$ is replaced by a class of affine function, $f$ is proper lsc convex iff it is $\Phi_{lsc}^\mathbb{R}$-convex \cite[Theorem 9.20]{Bau2011}.

Together with $\Phi_{lsc}^\mathbb{R}$-convexity, we present the definitions of $\Phi_{lsc}^\mathbb{R}$-conjugate and $\Phi_{lsc}^\mathbb{R}$-subdifferentials.
\begin{definition}[$\Phi_{lsc}^\mathbb{R}$-conjugate]
\label{def:lsc-conj}
Let $f:X\to (-\infty,+\infty]$ be proper, the $\lsc^\R$-conjugate of $f$ is defined as $f^*_\Phi :\Phi_{lsc}^\mathbb{R}\to\mathbb{R}$,
\begin{align}
\label{eq: def lsc conj}
f^*_\Phi (\phi) = \sup_{x\in X} \phi(x)-f(x),
\end{align}
and the $\Phi_{lsc}^\mathbb{R}$-biconjugate of $f$ defined on $X$ is
\begin{align}
\label{eq: def lsc biconj}
f^{**}_\Phi (x) = \sup_{\phi\in \Phi_{lsc}^\mathbb{R}} \phi(x)-f^*_\Phi (\phi).
\end{align}
\end{definition}
When $\Phi_{lsc}^\mathbb{R}$ is limited to the class of linear functions ($a=0$) and $X$ is a Banach space, we obtain Fenchel conjugation with the duality pairing $\langle \cdot,\cdot\rangle_{X^*\times X}$. In the case of $\Phi_{lsc}^\mathbb{R}$, we can equip $\Phi_{lsc}^\mathbb{R}$ with the norm and inner product as $\phi(x)=\langle (a,u), (-\Vert x\Vert^2,x)\rangle_{\Phi_{lsc}^\mathbb{R} \times (\mathbb{R}_-\times X)}$. 
Another formulation of abstract convexity is considering $\phi\in \Phi_{lsc}^\mathbb{R}$ as an element and $\phi(x)=c(x,\phi)$ with $c:X\times \Phi_{lsc}^\mathbb{R} \to \mathbb{R}$ as a coupling function e.g. \cite{laude2023anisotropic,laude2025anisotropic,oikonomidis2025forward}. 
We can also consider $\phi(x)$ as $x(\phi)$, so for the functions defined on $\Phi_{lsc}^\mathbb{R}$, we call them $X$-convex to distinguish with $\Phi_{lsc}^\mathbb{R}$-convex functions defined on $X$ \cite[Proposition 1.2.3]{Pall2013}.

Hence, $\Phi_{lsc}^\mathbb{R}$ plays the role of the "\emph{dual}" of $X$, similar to Fenchel conjugate.
From the definition, $f^*_\Phi (\phi)$ is convex on $\Phi_{lsc}^\mathbb{R}$ and so is $\Phi_{lsc}^\mathbb{R}$-convex, while $f^{**}_\Phi$ is $\Phi_{lsc}^\mathbb{R}$-convex on $X$. 

\begin{example}
Consider the function $h:X\to \mathbb{R}, h(x)=c \Vert x\Vert^2$, with $c\in\mathbb{R}$. The $\Phi_{lsc}^\mathbb{R}$-conjugate of $h$ at $\phi\in \Phi_{lsc}^\mathbb{R}$ is
\[
h^*_\Phi (\phi) = h^*_\Phi (a,u) =
\begin{cases}
0 & a=-c,u=0\\
\frac{\Vert u\Vert ^{2}}{4\left(a+c\right)} & a>-c\\
+\infty & \text{otherwise}.
\end{cases}
\]
\end{example}
If $a=0,c=1/2$, we recover the convex conjugate of $h(x)$ (see \cite[Example 13.6]{Bau2011}).
\begin{remark}
\
\begin{itemize}
    \item If $f$ is proper $\Phi_{lsc}^\mathbb{R}$-convex and $\sup_{ \phi=(a,u)\in \Phi_{lsc}^\mathbb{R}, \phi\leq f} a <+\infty$ then $\dom f$ is a convex set. Indeed, assume by contradiction, there exists $x,y\in \dom f$ and $\lambda\in[0,1]$ such that $f(\lambda x+(1-\lambda)y) = +\infty$. By the form of function $\phi\in\Phi_{lsc}^\mathbb{R}$, we have
    \begin{align*}
    +\infty =f(\lambda x+(1-\lambda)y) &= \sup_{\phi\leq f} \phi (\lambda x+(1-\lambda)y) \\
    &\leq \lambda \sup_{\phi\leq f} \phi(x)+ (1-\lambda) \sup_{\phi\leq f} \phi(x) +\lambda (1-\lambda) \sup_{\phi\leq f} a\Vert x-y\Vert^2 \\
    & = \lambda f(x)+(1-\lambda) f(y) +\lambda (1-\lambda) \sup_{\phi\leq f} a\Vert x-y\Vert^2.
    \end{align*}
    This implies $\sup_{\phi\leq f} a=+\infty$ which contradicts $\sup_{\phi\leq f} a <+\infty$.
    \item The convexity of $f^*_\Phi$ is a consequence of the fact that $\Phi_{lsc}^\mathbb{R}$ is a convex set. 
\end{itemize}
\end{remark}
Next, we define $\Phi_{lsc}^\mathbb{R}$-subdifferentials, see e.g. \cite{Pall2013,Rub2013}.
\begin{definition}[$\Phi_{lsc}^\mathbb{R}$-subdifferentials]
\label{def:lsc-subdiff}
Let $f:X\to (-\infty,+\infty]$ be proper, the $\lsc^\R$-subgradient of $f$ at $x_{0}\in\mathrm{dom }f$ is an element $\phi\in\lsc^\R$ such that
\begin{align}
\label{eq: def subdiff}
\left(\forall y\in X\right)\quad f\left(y\right)-f\left(x_{0}\right) & \geq\phi\left(y\right)-\phi\left(x_{0}\right).
\end{align}
We denote the collection of all $\phi$ satisfying \eqref{eq: def subdiff} as $\partial_{lsc}^\R f(x_0)$.
If $a=0$ in $\Phi_{lsc}^\mathbb{R}$, we have $\partial f$ as the subdifferentials in the sense of convex analysis. Similarly, we define the $\Phi_{lsc}^\mathbb{R}$-subgradient of $f^*_\Phi:\Phi_{lsc}^\mathbb{R}\to \mathbb{R}$ at $\phi_0$ is an element $x\in X$ which satisfies
\begin{equation}
\label{eq: lsc subgrad conj}
    (\forall \phi \in \Phi_{lsc}^\mathbb{R}) \quad f^*_\Phi (\phi)- f^*_\Phi (\phi_0) \geq \phi(x)-\phi_0 (x).
\end{equation}
The collection of $X$-subgradient of $f^*_\Phi$ at $\phi_0$ is called $X$-subdifferentials and is denoted by $\partial_X f^*_\Phi (\phi_0)$.
\end{definition}
The properties of $\Phi_{lsc}^\mathbb{R}$-subgradient also analogous to convex subgradient.

\begin{proposition}
\label{prop: lsc subgrad prop}
Let $f:X\to (-\infty,+\infty]$ be proper, $x\in \dom f$ and $\phi\in \Phi_{lsc}^\mathbb{R}$, we have
\begin{enumerate}[label=\roman*.]
    \item\label{prop: lsc subgrad prop-1} If $\partial_{lsc}^\mathbb{R} f(x)\neq \emptyset$ then $f(x)=f^{**}_\Phi(x)$. As a consequence, $f$ is $\Phi_{lsc}^\mathbb{R}$-convex on its domain if $\partial_{lsc}^\mathbb{R} f(x)\neq \emptyset$ for all $x\in \dom f$ and $\dom f$ is convex.
    \item $f(x)+f^*_\Phi(\phi) = \phi(x)$ for all $\phi\in \partial_{lsc}^\mathbb{R} f(x)$.
    \item For $x\in X$, $0 \in \partial_{lsc}^\mathbb{R} f(x)$ if and only if $x \in \mathrm{argmin} f$.
    \item If $f$ is $\Phi_{lsc}^\mathbb{R}$-convex, then $\phi\in \partial_{lsc}^\mathbb{R} f(x)\Leftrightarrow x\in \partial_X f^*_\Phi (\phi)$.
\end{enumerate}
\end{proposition}
The proof can be found in \cite{Rub2013,Pall2013,bednarczuk2023duality} for general class $\Phi$. If $a=0$, then Proposition \ref{prop: lsc subgrad prop}-\ref{prop: lsc subgrad prop-1} can be proved in \cite[Theorem 2.4.1-(iii)]{zalinescu2002convex}.
\begin{example}
\label{ex:1 J function}
Let $\gamma>0$, we define the function $g_\gamma :X\to \mathbb{R}, g_\gamma(x) =\frac{1}{2\gamma} \Vert x\Vert^2$, the $\Phi_{lsc}^\mathbb{R}$-subgradient of $g_\gamma$ at $x_0$ is 
\begin{equation}
\partial_{lsc}^\R (\frac{1}{2\gamma}\Vert \cdot\Vert^2) \left(x_{0}\right)=\left\{ \phi\in\lsc^\R: \phi =\left(a,\left(\frac{1}{\gamma}+2a\right)x_{0}\right),\ 2\gamma a\geq -1\right\}.
\label{eq: J_gamma form}
\end{equation}
\end{example}
The non-emptiness of the interior $\dom f$ does not guarantee that $\partial_{lsc}^\R f(x) \neq \emptyset$ for $x\in \dom f$ even though $f$ is $\Phi_{lsc}^\mathbb{R}$-convex. One typical example from \cite{syga2019global}, the function $f(x)=-|x|^{\frac{3}{2}}$ which has $\partial_{lsc}^\mathbb{R}f(0)=\emptyset$.

\subsection{\texorpdfstring{$\Phi_{lsc}^\mathbb{R}$-}\ Proximal Operator}
\label{sec: prox operator}
In this subsection, let $X$ be a Hilbert space, we introduce the proximal operator related to the class $\lsc^\R$-convex functions.
Observe that $\lsc^\R$-subgradient is defined on $\lsc^\R$ which is not a subset of $X$. 
Therefore, we need a mapping which can serve as a link between $X$ and $\lsc^\R$. In analogy to the classical constructions we consider the function $g_\gamma$ from Example \ref{ex:1 J function}. 
\begin{definition}
\label{def: Jgamma duality map}
Let $\gamma>0$ and $g_\gamma(x) = \frac{1}{2\gamma}\Vert x\Vert^2$, we define \emph{$\lsc^\R$-duality map} $J_\gamma: X \rightrightarrows \lsc^\R$ as  
\begin{equation*}
    J_\gamma (x) := \partial_{lsc}^\R g_\gamma (x) =\partial_{lsc}^\R \left( \frac{1}{2\gamma}\Vert x \Vert^2 \right).
\end{equation*} 
Its inverse $J^{-1}_\gamma :\lsc^\R\rightrightarrows X$ is
\begin{equation*}
    J^{-1}_\gamma (\phi) = (\partial_{lsc}^\R g_\gamma)^{-1} (\phi).
\end{equation*} 
\end{definition}
When $\gamma=1$, $\partial g_\gamma(x)$ is the classical duality mapping known in \nolinebreak\cite[Example 2.26]{phelps2009convex}.
Next, we define $\Phi_{lsc}^\mathbb{R}$-proximal operator.
\begin{definition}
\label{def: lsc prox}
Let $f:X\to (-\infty,+\infty]$ and $\gamma>0$, the $\lsc^\R$-proximal operator $\prox^{\mathrm{lsc}}_{\gamma f}: X\rightrightarrows X$ is a set-valued mapping of the form,
\begin{equation}
\proxlsc_{\gamma f} (x) := \left( J_\gamma + \partial_{lsc}^\R f\right)^{-1} J_\gamma (x),
\label{eq:proximal-like operator 1}
\end{equation}
where $J_\gamma$ is defined in Definition \ref{def: Jgamma duality map}.
\end{definition}
The idea of $\proxlsc_{\gamma f}$ is related to the resolvent operator which is defined for classical convex subdifferentials as $(Id+\partial f)^{-1}$. When $f$ is convex, this is also known as proximity operator
\begin{equation}
\label{eq: proximal operator}
\mathrm{prox}_{\gamma f}\left(x_{0}\right)=\arg\min_{z\in X}f\left(z\right)+\frac{1}{2\gamma}\left\Vert z-x_{0}\right\Vert ^{2}.
\end{equation}
We state some properties of $\Phi_{lsc}^\mathbb{R}$-proximal operator.
\begin{proposition}\cite[Theorem 3, Theorem 4]{bednarczuk2025proximal}
\label{prop: lsc prox properties}
    Let $X$ be a Hilbert space. Let $f:X\to(-\infty,+\infty]$ be a proper $\lsc^\R$-convex function and $\gamma>0$.  
    The following holds:
    \begin{enumerate}[label=\roman*.]
        \item If $x_0$ is a global minimizer of $f$ then $x_0$ is a fixed point of $\proxlsc_{\gamma f}$.
        \item\label{prop: lsc prox properties-2} Let $x_0\in \dom f$, $x \in \proxlsc_{\gamma f} (x_0) $ if and only if there exists $a_0 \geq -1/2\gamma$ such that 
        \begin{equation}
        \label{eq: lsc prox equiv prox}
            x\in \arg\min_{z \in X} \left[ f(z) +\left( \frac{1}{2\gamma} +a_0 \right) \Vert z-x_0\Vert^2 \right] 
        \end{equation}
    \end{enumerate}
\end{proposition}
\begin{remark}
\label{rmk: prox form and lsc form}
Proposition \ref{prop: lsc prox properties}-(ii) gives us a way to calculate $\Phi_{lsc}^\mathbb{R}$-proximal operator. However, the general formulation of $\Phi_{lsc}^\mathbb{R}$-proximal operator requires us to solve the problem: given $x_0 \in \dom f$, $x\in \proxlsc_{\gamma f} (x_0)$ is equivalent to solving the following relation
\begin{equation}
\label{eq: lsc prox in operator form}
\text{Find $x\in X$ such that }\ J_\gamma(x_0)-J_\gamma (x) \cap \partial_{lsc}^\mathbb{R} f(x) \neq \emptyset.
\end{equation}
The value of $\phi\in J_\gamma (x)$ will play the crucial role in the convergence of the primal-dual algorithms.


\end{remark}

When $f$ is $\Phi_{lsc}^\mathbb{R}$-convex, we can explicitly calculate $\proxlsc_{\gamma f} (x_0)$ for $x_0\in\dom f$ under additional assumption.

\begin{proposition}
\label{prop: lsc prox with lsc convex}
Let $f:X\to(-\infty,+\infty]$ be proper $\Phi_{lsc}^\mathbb{R}$-convex Assume that $\dom f = \dom \partial_{lsc}^\mathbb{R} f$. 
For any $x_0\in \dom f$, if there exists $x\in X$ such that \eqref{eq: lsc prox equiv prox} holds then \eqref{eq: lsc prox in operator form} holds.
\end{proposition}
\begin{proof}
Let $x_0\in \dom f$ and $a_0 \geq -\frac{1}{2\gamma}$, and assume the \eqref{eq: lsc prox equiv prox} holds, then for any $z\in X$
\[
f(z) -f(x) \geq \left(\frac{1}{2\gamma}+a_0\right)\left( \Vert x-x_0\Vert^2 -\Vert z-x_0\Vert^2\right).
\]
Let $z=\lambda z_1 +(1-\lambda) x \in \dom f$ for any $\lambda\in(0,1], z_1\in X$, the right side of the above inequality becomes
\begin{align}
\Vert x-x_0\Vert^2 -\Vert \lambda z_1 +(1-\lambda) x -x_0\Vert^2 & =  \lambda^2\Vert z_1-x\Vert^2 + 2\lambda\langle x_0-x,z_1-x\rangle.
\label{eq: lsc prox expl-1}
\end{align}
On the other hand,
since $\dom f=\dom \partial_{lsc}^\mathbb{R}f$, there exists $\phi_z \leq f$ such that $f(z)=\phi_z(z)$, we notice 
\[
-a \Vert z\Vert^2 +\langle u_z,z\rangle \leq -a_z\Vert z\Vert^2 +\langle u_z,z\rangle = f(z), 
\]
for any $a\geq a_z$. Let us denote the set of the smallest value $a_z$ as $A(z)= \{ \inf a_z: \phi_z(z) = f(z)\}$ which is finite. Then we can take 
\[
a_\phi = \sup_{z \in \dom f} A(z) <+\infty.
\]
We express the $\Phi_{lsc}^\mathbb{R}$-convexity of $f$ at $z$,
\begin{align}
f(\lambda z_1+(1-\lambda)x) &=  \phi_z (\lambda z_1+(1-\lambda)x) \nonumber\\ 
& = \lambda \phi_z (z_1) + (1-\lambda) \phi_z (x) +\lambda(1-\lambda) \Vert z_1-x\Vert ^2 a_{\phi_z} \nonumber\\
& \leq \lambda f(z_1)+(1-\lambda) f(x) +\lambda(1-\lambda) \Vert z_1-x\Vert ^2 a_{\phi}.
\label{eq: lsc prox expl-2}
\end{align}
Combining \eqref{eq: lsc prox expl-1} and \eqref{eq: lsc prox expl-2}, we obtain 
\begin{equation*}
\lambda (f(z_1)-f(x)) + \lambda(1-\lambda) \Vert z_1-x\Vert ^2  a_{\phi} \geq \left(\frac{1}{2\gamma}+a_0\right) (\lambda^2\Vert z_1-x\Vert^2 + 2\lambda\langle x_0-x,z_1-x\rangle).
\end{equation*}
Dividing both sides by $\lambda$ and let $\lambda\to 0$, we have
\begin{align}
\label{eq: lsc prox expl-3}
f(z_1)-f(x) + \Vert z_1-x\Vert ^2 a_{\phi} \geq \left(\frac{1}{\gamma}+2a_0\right) \langle x_0-x,z_1-x\rangle.
\end{align}
Taking $a_0-a \geq a_\phi$ which is well-defined by the assumption, then \eqref{eq: lsc prox expl-3} becomes
\begin{align}
f(z_1)-f(x) & \geq -(a_0-a)\Vert z_1-x\Vert ^2 + \left(\frac{1}{\gamma}+2a_0\right) \langle x_0-x,z_1-x\rangle \nonumber\\
& = -(a_0-a)(\Vert z_1\Vert ^2 - \Vert x\Vert ^2) + \langle \left(\frac{1}{\gamma}+2a_0\right) x_0- \left(\frac{1}{\gamma}+2a\right)x,z_1-x\rangle.
\end{align}
This implies $\phi_0-\phi \in \partial_{lsc}^\mathbb{R} f(x)$ where $\phi_0 = (a_0, (1/\gamma+2a_0)x_0) \in J_\gamma (x_0)$ so we have \eqref{eq: lsc prox in operator form}.
\end{proof}
Therefore, in the analysis below, we use \eqref{eq: lsc prox in operator form} while writing \eqref{eq: lsc prox equiv prox} as a practical way to compute $\lsc$-proximal operator.

\section{\texorpdfstring{$\Phi_{lsc}^\mathbb{R}$-}\ Primal-Dual Algorithms}
\label{sec: pd alg lsc prox}
To design an effective algorithm for solving \eqref{prob:DCP}, we first investigate primal-dual algorithm for the case $g$ convex before tackling a more general case of $g$.
In this section, we apply our newly defined $\lsc^\R$-proximal operator \eqref{eq:proximal-like operator 1} to solve the saddle point problem of the form
\begin{equation}
\label{prob: Lagrange g convex}
\inf_{x\in X}\sup_{y\in Y} \mathcal{L}(x,y):= \inf_{x\in X}\sup_{y\in Y} f(x)+ \langle Lx,y\rangle -g^*(y),
\end{equation}
where $g^*(y)$ is the convex conjugate of $g$.
When $g$ is lsc convex, problem \eqref{prob: Lagrange g convex} is equivalent to the composite problem \cite{Bed2020},
\begin{equation}
    \label{prob:CP}
    \min_{x\in X} f(x)+g(Lx) \tag{CP}.
\end{equation}
A pair $(\hat{x},\hat{y})\in X\times Y$ is a saddle point of problem \eqref{prob: Lagrange g convex} if and only if the following are satisfied
\begin{equation}
\label{eq: saddle point def}
    (\forall (x,y) \in X\times Y) \quad \mathcal{L}( {x}, \hat{y}) \geq \mathcal{L}(\hat{x},\hat{y}) \geq \mathcal{L} ( \hat{x}, {y}).
\end{equation}
We denote by $S$ the set of all the saddle points. 
Equivalently, \eqref{eq: saddle point def} can be rewritten as
\begin{align*}
    (\forall x\in X) \quad f(x) - f(\hat{x}) & \geq -\langle L^*\hat{y}, x-\hat{x} \rangle_{X\times X}, \\
    (\forall y\in Y) \quad g^*(y) - g^*(\hat{y}) & \geq \langle L \hat{x}, y-\hat{y} \rangle_{Y\times Y}.
\end{align*}
Since $g^*$ is convex and $f$ is $\Phi_{lsc}^\mathbb{R}$-convex, we can rewrite the above inequalities in the form of subdifferentials 
\begin{equation}
    L\hat{x} \in \partial g^* (\hat{y}), \qquad
    (0,-L^*\hat{y}) \in \partial_{lsc}^{\R} f(\hat{x}),\label{eq: kkt condition lsc convex}
\end{equation}
which is known as KKT condition for the Lagrangian primal-dual problems \eqref{prob: Lagrange g convex}. 
\begin{remark}
When we replace the convexity of $g$ by $\Psi$-convexity with a more general collection of function $\Psi=\{ \psi:Y\to\mathbb{R}\}$ and obtain the abstract saddle point problem
\begin{equation}
\label{prob: Lagrange g psi convex}
\inf_{x\in X}\sup_{\psi\in \Psi} \mathcal{L}(x,\psi):= \inf_{x\in X}\sup_{\psi\in \Psi} f(x)+ \psi(Lx) -g^*_\Psi(\psi).
\end{equation}
The duality theory for problem \eqref{prob: Lagrange g psi convex} has been investigated in \cite{bednarczuk2023duality}. 
In fact, zero duality gap guarantees the existence of $\varepsilon$-saddle point (in the sense of \cite[Theorem 4.2]{bednarczuk2023duality}
\end{remark}

We state the main assumptions for the construction and convergence of Algorithm \ref{alg: lsc Chambolle-Pock}.
\begin{assumption}
\label{ass: 1 g convex}
\
\begin{enumerate}[label=(\roman*)]
    \item $X,Y$ are Hilbert spaces.
    \item The function $f: X\to (-\infty,+\infty]$ is proper $\lsc^\R$-convex, and $g:Y\to (-\infty,+\infty]$ is proper lsc convex.
    \item $L:X\to Y$ is a bounded linear operator with its adjoint $L^*:Y\to X$.
    \item The set of saddle points to the problem \eqref{prob: Lagrange g convex} $S$ is nonempty.
    \item\label{ass: prox well defined} $\dom f = \dom \partial_{lsc}^\mathbb{R} f$ and for every $a_0 >-1/2\sigma$, \eqref{eq: lsc prox equiv prox} holds.
\end{enumerate}
\end{assumption}


Inspired by Chambolle-Pock Algorithm \cite[Algorithm 1]{chambolle2011first}, we propose Algorithm \ref{alg: lsc Chambolle-Pock} to solve Lagrange saddle point problem \eqref{prob: Lagrange g convex}.
By replacing the primal update of \cite[Algorithm 1]{chambolle2011first} with $\lsc^\R$-proximal update below, we call this new algorithm $\lsc^\R$-CP Algorithm and show that we can obtain convergence towards a saddle point.

\begin{algorithm}[H]\caption{$\lsc^\R$-Chambolle-Pock Algorithm}\label{alg: lsc Chambolle-Pock}
\textbf{Initialize:} Choose $\tau,\sigma >0, (x_0,y_0)\in X\times Y$ and $\bar{x}_0 = x_0$\\
{\normalfont\textbf{Update: }}{
For $n\in\N$,
\begin{itemize}
    \item Dual step update:
    \begin{itemize}
        \item[$\blacksquare$] $y_{n+1}  =\argmind{y\in Y}{ g^*(y) +\frac{1}{2\tau} \Vert y - y_n -\tau L \bar{x}_{n}\Vert^2}$ 
    \end{itemize}
    \item Primal step update: 
    \begin{itemize}
        \item[$\blacksquare$] Pick $\phi_n = \left( a_n, (\frac{1}{\sigma}+2a_n)x_n\right) \in J_\sigma (x_n)$ according to \eqref{eq: J_gamma form}
        \item[$\blacksquare$] \textbf{If} $a_n <-1/2\sigma$. \quad \textbf{Stop the Algorithm.}
        \item[$\blacksquare$] Pick $x_{n+1} \in \argmind{z\in X}{ f(z) + \langle L^* y_{n+1},z \rangle +\left( \frac{1}{2\sigma} +a_n\right) \Vert z-x_n\Vert^2}$
    \end{itemize}
\item $\bar{x}_{n+1} = 2x_{n+1} -x_{n}$ 
\end{itemize}}
\end{algorithm}
\vspace{0.5cm}
The dual iterates $(y_n)_{n\in\mathbb{N}}$ is updated in the same manner as Chambolle-Pock algorithm for convex case, while the primal iterates $(x_n)_{n\in\mathbb{N}}$ is generated by $\Phi_{lsc}^\mathbb{R}$-proximal operation i.e. we update $x_{n+1}$ such that (see Remark \ref{rmk: prox form and lsc form})
\begin{equation}
\label{eq: alg1 primal update subdiff form}
J_\gamma(x_n)-J_\gamma (x_{n+1}) \in \partial_{lsc}^\mathbb{R} (f+\langle L^* y_{n+1},\cdot\rangle) (x_{n+1}).
\end{equation}
The above relation is well-defined thanks to the Assumption \ref{ass: 1 g convex}-\ref{ass: prox well defined}, i.e. (see Proposition \ref{prop: lsc prox with lsc convex}). 
Hence, we guarantee the existence of $x_{n+1}$. However, notice that the condition $a_n \geq -1/2\sigma$ from $J_\gamma(x_n)$ for all $n\in\mathbb{N}$ will limit our choice of constructing the sequence $(a_{n})_{n\in\mathbb{N}}$.

We present our main convergence results in the next theorem.

\begin{theorem}
\label{thm: pd lsc prox convergence}
Let $f:X\to (-\infty,+\infty]$ be a proper $\lsc^\R$-convex, $g:Y\to (-\infty,+\infty]$ be a proper lsc convex function and $L:X\to Y$ be a bounded linear operator as in Assumption \ref{ass: 1 g convex}.
Let $(x_n)_{n\in\N},(y_n)_{n\in\N}, (\bar{x}_{n})_{n\in\N}$ and $(a_n)_{n\in\N}$ be the sequences generated by Algorithm \ref{alg: lsc Chambolle-Pock}. Assume that $\tau\sigma \Vert L\Vert^2 <1$ and $1+2\sigma a_n > \sqrt{\sigma \tau}\| L \|$ for all $n\in \N$.
Then the followings hold.
\begin{enumerate}[label=\roman*.]
    \item For any $(x,y)\in X\times Y$, 
    \begin{align}
& \sum_{n=0}^N \mathcal{L}\left(x,y_{n+1}\right)-\mathcal{L}\left(x_{n+1},y\right) \nonumber\\
&\geq  \frac{1-\sqrt{\sigma\tau}\Vert L\Vert}{2\tau} \Vert y-y_{N+1}\Vert^2 - \frac{1}{2\tau } \Vert y-y_{0}\Vert^2 \nonumber\\
& +\left(\frac{1}{2\sigma}+a_{N+1}\right) \Vert x-x_{N+1}\Vert^2 - \left(\frac{1}{2\sigma}+a_{0}\right) \Vert x-x_0\Vert^2, \label{eq: lsc lagrange summable final form sum}
\end{align}
    where $\mathcal{L} (x,y) = f(x)+\langle Lx,y\rangle -g^* (y)$ is the Lagrangian.
    \item The sequences $(x_n,y_n)_{n\in\N}$ are bounded.
    \item\label{thm: CP lsc-convex convergence-3} Let $X,Y$ be finite dimensional. If { $a_n-a_{n+1}\to 0$ and $1+2\sigma a_n >\sqrt{\sigma\tau}\Vert L\Vert$ for all $n\in\mathbb{N}$}, then $(x_n,y_n)_{n\in\N}$ converges to a saddle point in the sense of \eqref{eq: kkt condition lsc convex}.
\end{enumerate}
\end{theorem}
\begin{proof}
For $(i)$, from the update of Algorithm \ref{alg: lsc Chambolle-Pock}, we have
\begin{align}
\frac{y_{n}-y_{n+1}}{\tau}+L \bar{x}_{n}&\in\partial g^* \left(y_{n+1}\right), \nonumber\\
\phi_n -\phi_{n+1} &\in \partial_{lsc}^\R (f(\cdot)+\langle L^* y_{n+1},\cdot\rangle) (x_{n+1}), \label{eq: lsc alg update subgrad}
\end{align}
where $\phi_n$ comes from $J_\gamma(x_n)$.
By \eqref{eq: lsc alg update subgrad}, we obtain the following estimation, 
\begin{align}
& \mathcal{L}\left(x,y_{n+1}\right)-\mathcal{L}\left(x_{n+1},y\right) \nonumber\\
&\geq  \frac{1}{2\tau} \left[ \Vert y-y_{n+1}\Vert^2 + \Vert y_n-y_{n+1}\Vert^2 - \Vert y-y_{n}\Vert^2\right] \nonumber\\
& +\left(\frac{1}{2\sigma}+a_{n+1}\right) \Vert x-x_{n+1}\Vert^2 + \left(\frac{1}{2\sigma}+a_{n}\right)\left( \Vert x_n -x_{n+1}\Vert^2 - \Vert x-x_n\Vert^2 \right) \nonumber\\
& + \langle L \bar{x}_n, y-y_{n+1}\rangle -\langle y_{n+1},L (x-x_{n+1})\rangle + \langle y_{n+1},Lx\rangle - \langle y,Lx_{n+1}\rangle. \label{eq: CPock-lsc 2nd estimation Lagrange}
\end{align}
By using $\bar{x}_n = 2x_n - x_{n-1}$ in the last line of \eqref{eq: CPock-lsc 2nd estimation Lagrange},
\begin{align}
    & \langle L \bar{x}_n, y-y_{n+1}\rangle -\langle y_{n+1},L (x-x_{n+1})\rangle + \langle y_{n+1},Lx\rangle - \langle y,Lx_{n+1}\rangle \nonumber\\
    =& \langle L(2x_n -x_{n-1} -x_{n+1}), y-y_{n+1}\rangle \nonumber \\
    =& \langle L(x_n -x_{n+1}), y-y_{n+1}\rangle - \langle L(x_{n-1} -x_{n}), y-y_{n}\rangle + \langle L(x_n -x_{n-1}), y_n-y_{n+1}\rangle \nonumber\\
    \geq & \langle L(x_n -x_{n+1}), y-y_{n+1}\rangle - \langle L(x_{n-1} -x_{n}), y-y_{n}\rangle \nonumber\\
    & -\frac{\sqrt{\sigma \tau}\Vert L\Vert}{2\tau} \Vert y_n-y_{n+1}\Vert^2 - \frac{\sqrt{\sigma \tau}\Vert L\Vert}{2\sigma} \Vert x_n-x_{n-1}\Vert^2. \label{eq: CPock-lsc 2xn -xn-1}
\end{align}
Plugging \eqref{eq: CPock-lsc 2xn -xn-1} back into \eqref{eq: CPock-lsc 2nd estimation Lagrange}, for every $n\in\mathbb{N}$, we have
\begin{align}
& \mathcal{L}\left(x,y_{n+1}\right)-\mathcal{L}\left(x_{n+1},y\right) \nonumber\\
&\geq  \frac{1}{2\tau} \left[ \Vert y-y_{n+1}\Vert^2 - \Vert y-y_{n}\Vert^2\right] + \frac{1-\sqrt{\sigma\tau}\Vert L\Vert}{2\tau}\Vert y_n-y_{n+1}\Vert^2  \nonumber\\
& +\left(\frac{1}{2\sigma}+a_{n+1}\right) \Vert x-x_{n+1}\Vert^2 - \left(\frac{1}{2\sigma}+a_{n}\right) \Vert x-x_n\Vert^2 \nonumber\\
& +\langle L(x_n -x_{n+1}), y-y_{n+1}\rangle - \langle L(x_{n-1} -x_{n}), y-y_{n}\rangle \nonumber\\
& + \left(\frac{1}{2\sigma}+a_{n}\right)\Vert x_n -x_{n+1}\Vert^2 - \frac{\sqrt{\sigma \tau}\Vert L\Vert}{2\sigma} \Vert x_n-x_{n-1}\Vert^2.\label{eq: lsc lagrange summable form}
\end{align}
Summing up \eqref{eq: lsc lagrange summable form} from zero to $N\in\N$, we obtain
\begin{align}
& \sum_{n=0}^N \mathcal{L}\left(x,y_{n+1}\right)-\mathcal{L}\left(x_{n+1},y\right) \nonumber\\
&\geq  \frac{1}{2\tau} \left[ \Vert y-y_{N+1}\Vert^2 - \Vert y-y_{0}\Vert^2\right] + \sum_{n=0}^N\frac{1-\sqrt{\sigma\tau}\Vert L\Vert}{2\tau}\Vert y_n-y_{n+1}\Vert^2  \nonumber\\
& +\left(\frac{1}{2\sigma}+a_{N+1}\right) \Vert x-x_{N+1}\Vert^2 - \left(\frac{1}{2\sigma}+a_{0}\right) \Vert x-x_0\Vert^2 \nonumber\\
& + \left(\frac{1}{2\sigma}+a_{N}\right)\Vert x_N -x_{N+1}\Vert^2 +\sum_{n=0}^{N-1} \left(\frac{1-\sqrt{\sigma \tau}\Vert L\Vert }{2\sigma}+a_{n-1} \right) \Vert x_n-x_{n-1}\Vert^2 \nonumber\\
& +\langle L(x_N -x_{N+1}), y-y_{N+1}\rangle 
.\label{eq: lsc lagrange summable form sum}
\end{align}
By applying Young's inequality to the last term of \eqref{eq: lsc lagrange summable form sum}, we have 
\begin{equation*}
   \langle L(x_N -x_{N+1}), y-y_{N+1}\rangle \geq -\frac{\sqrt{\sigma\tau} \Vert L\Vert}{2\tau } \Vert y-y_{N+1}\Vert^2 -\frac{\sqrt{\sigma\tau} \Vert L\Vert}{2\sigma} \Vert x_N -x_{N+1}\Vert^2.
\end{equation*}
Hence,
\begin{align}
& \sum_{n=0}^N \mathcal{L}\left(x,y_{n+1}\right)-\mathcal{L}\left(x_{n+1},y\right) \nonumber\\
&\geq  \frac{1-\sqrt{\sigma\tau}\Vert L\Vert}{2\tau} \Vert y-y_{N+1}\Vert^2 - \frac{1}{2\tau } \Vert y-y_{0}\Vert^2 + \sum_{n=0}^N\frac{1-\sqrt{\sigma\tau}\Vert L\Vert}{2\tau}\Vert y_n-y_{n+1}\Vert^2  \nonumber\\
& +\left(\frac{1}{2\sigma}+a_{N+1}\right) \Vert x-x_{N+1}\Vert^2 - \left(\frac{1}{2\sigma}+a_{0}\right) \Vert x-x_0\Vert^2 \nonumber\\
& +\sum_{n=0}^{N} \left(\frac{1-\sqrt{\sigma \tau}\Vert L\Vert }{2\sigma}+a_{n} \right) \Vert x_n-x_{n+1}\Vert^2. \label{eq: lsc prox pd finite sum with n-1}
\end{align}
By assumption, $\sqrt{\sigma\tau} \Vert L\Vert <1$ and $1+ 2\sigma a_n > \sqrt{\sigma\tau} \Vert L\Vert$ for all $n\in\N$, we have
\begin{align*}
\sum_{n=0}^N \mathcal{L}\left(x,y_{n+1}\right)-\mathcal{L}\left(x_{n+1},y\right)
&\geq  \frac{1-\sqrt{\sigma\tau}\Vert L\Vert}{2\tau} \Vert y-y_{N+1}\Vert^2 - \frac{1}{2\tau } \Vert y-y_{0}\Vert^2 \\
& +\left(\frac{1}{2\sigma}+a_{N+1}\right) \Vert x-x_{N+1}\Vert^2 - \left(\frac{1}{2\sigma}+a_{0}\right) \Vert x-x_0\Vert^2,
\end{align*}
which is \eqref{eq: lsc lagrange summable final form sum}.

For $(ii)$, since \eqref{eq: lsc lagrange summable final form sum} holds for all $(x,y)\in X\times Y$, we can take $(x,y) = (x^*,y^*)\in S$. By \eqref{eq: saddle point def}, we have
\begin{align*}
& 0\geq \sum_{n=0}^N \mathcal{L}\left(x^*,y_{n+1}\right)-\mathcal{L}\left(x_{n+1},y^*\right)\\
&\geq  \frac{1-\sqrt{\sigma\tau}\Vert L\Vert}{2\tau} \Vert y^*-y_{N+1}\Vert^2 - \frac{1}{2\tau } \Vert y^*-y_{0}\Vert^2 \\
& +\left(\frac{1}{2\sigma}+a_{N+1}\right) \Vert x^*-x_{N+1}\Vert^2 - \left(\frac{1}{2\sigma}+a_{0}\right) \Vert x^*-x_0\Vert^2.
\end{align*}
Since the above inequality holds for any $N\in\mathbb{N}$, we obtain that $(x_n)_{n\in\N},(y_n)_{n\in\N}$ are bounded.

Lastly, we prove $(iii)$. Let $X,Y$ be finite dimensional. In view of the above, let $(x_{n_k})_{k\in\N},(y_{n_k})_{k\in\N}$ be the subsequences of $(x_{n})_{n\in\N},(y_{n})_{n\in\N}$ which converge to $\overline{x}$ and $\overline{y}$, respectively. By \eqref{eq: lsc prox pd finite sum with n-1}, with $(x,y)=(x^*,y^*)$ as a saddle point, we have
\begin{align*}
0 & \geq \sum_{n=0}^N \mathcal{L}\left(x^*,y_{n+1}\right)-\mathcal{L}\left(x_{n+1},y^*\right) \nonumber\\
&\geq  \frac{1-\sqrt{\sigma\tau}\Vert L\Vert}{2\tau} \Vert y^*-y_{N+1}\Vert^2 - \frac{1}{2\tau } \Vert y^*-y_{0}\Vert^2 + \sum_{n=0}^N\frac{1-\sqrt{\sigma\tau}\Vert L\Vert}{2\tau}\Vert y_n-y_{n+1}\Vert^2  \nonumber\\
& +\left(\frac{1}{2\sigma}+a_{N+1}\right) \Vert x^*-x_{N+1}\Vert^2 - \left(\frac{1}{2\sigma}+a_{0}\right) \Vert x^*-x_0\Vert^2 \nonumber\\
& +\sum_{n=0}^{N} \left(\frac{1-\sqrt{\sigma \tau}\Vert L\Vert }{2\sigma}+a_{n} \right) \Vert x_n-x_{n+1}\Vert^2 \\
& \geq  \sum_{n=0}^N\frac{1-\sqrt{\sigma\tau}\Vert L\Vert}{2\tau}\Vert y_n-y_{n+1}\Vert^2 +\sum_{n=0}^{N} \left(\frac{1-\sqrt{\sigma \tau}\Vert L\Vert }{2\sigma}+a_{n} \right) \Vert x_n-x_{n+1}\Vert^2 \nonumber\\
& - \frac{1}{2\tau } \Vert y^*-y_{0}\Vert^2 - \left(\frac{1}{2\sigma}+a_{0}\right) \Vert x^*-x_0\Vert^2 \nonumber.
\end{align*}
Letting $N\to\infty$, we have the two sums above are finite. This means
\[
\Vert y_n-y_{n+1}\Vert^2 \to 0, \quad \left(\frac{1-\sqrt{\sigma \tau}\Vert L\Vert }{2\sigma}+a_{n} \right) \Vert x_n-x_{n+1}\Vert^2 \to 0.
\]

By assumption, $1+ 2\sigma a_n > \sqrt{\sigma\tau} \Vert L\Vert$ for all $n\in\N$ and $(a_n)_{n\in\N}$ converges, we infer
\[
\lim_{n\to\infty} \Vert x_n-x_{n+1} \Vert = \lim_{n\to\infty} \Vert y_n - y_{n+1} \Vert =0.
\]

Then $x_{n_k+1} \to \overline{x}$ and $y_{n_k+1} \to \overline{y}$ as well. Now we want to prove that $(\overline{x},\overline{y})$ is a saddle point. By \eqref{eq: lsc alg update subgrad}, we have, for every $y\in Y$,
\begin{equation*}
    g^*(y) -g^*(y_{n_k+1}) \geq \langle \frac{y_{n_k}-y_{n_k+1}}{\tau}+L \bar{x}_{n_k}, y-y_{n_k+1}\rangle.
\end{equation*}
Since $g^*$ is lsc convex \cite[Proposition 13.13]{Bau2011}, letting $k\to \infty$, we have, for every $y\in Y$,
\begin{align}
g^*(y) -g^*( \overline{y}) & \geq \liminf_{k\to\infty} g^*(y) -g^*(y_{n_k+1}) \nonumber\\
& \geq \liminf_{k\to\infty} \ \langle \frac{y_{n_k}-y_{n_k+1}}{\tau}+L \bar{x}_{n_k}, y-y_{n_k+1}\rangle \nonumber\\
& = \langle L \overline{x}, y-\overline{y}\rangle.
\label{eq: saddle point y bar}
\end{align}

Similarly for lsc $f$, by \eqref{eq: lsc prox pd finite sum with n-1}, for every $x\in X$,
\begin{align*}
f(x) -f(x_{n_k+1}) &\geq (\phi_{n_k} -\phi_{n_k+1})(x) - (\phi_{n_k} -\phi_{n_k+1})(x_{n_k+1}) \\
& - \langle L^* y_{n+1},x-x_{n_k+1}\rangle \\
& = \left(\frac{1}{2\sigma}+a_{n_k+1}\right) \Vert x-x_{n_k+1}\Vert^2 - \left(\frac{1}{2\sigma}+a_{n_k}\right) \Vert x-x_{n_k}\Vert^2 \\
& + \left(\frac{1}{2\sigma}+a_{n_k}\right) \Vert x_{n_k+1}-x_{n_k}\Vert^2 - \langle L^*y_{n_k+1}, x-x_{n_k+1}\rangle,
\end{align*}
where $\phi_n = (a_n,(1/\sigma + 2a_n)x_n) \in \lsc^\R$, the second equality is a direct simplification of $\phi_n$.

Letting $k\to \infty$ and note that $a_n-a_{n+1}\to 0$, we obtain
\begin{align}
f(x) -f(\overline{x}) & \geq \liminf_{k\to\infty} f(x) -f(x_{n_k+1}) \nonumber \\
& \geq \liminf_{k\to\infty} \left[ \left(\frac{1}{2\sigma}+a_{n_k+1}\right) \Vert x-x_{n_k+1}\Vert^2 - \left(\frac{1}{2\sigma}+a_{n_k}\right) \Vert x-x_{n_k}\Vert^2 \right. \nonumber\\
& \left. + \left(\frac{1}{2\sigma}+a_{n_k}\right) \Vert x_{n_k+1}-x_{n_k}\Vert^2 - \langle L^*y_{n_k+1}, x-x_{n_k+1}\rangle, \right] \nonumber\\
& \geq - \langle L^*\overline{y}, x-\overline{x}\rangle.
\label{eq: saddle point x bar}
\end{align}
As \eqref{eq: saddle point x bar} and \eqref{eq: saddle point y bar} hold for all $(x,y)\in X\times Y$, this means that $(\overline{x},\overline{y})$ satisfies the KKT condition \eqref{eq: kkt condition lsc convex} and so it is a saddle point. 

Next, we prove that the sequence $(x_n,y_n)_{n\in \N}$ converges to $(\overline{x},\overline{y})$. Let us consider the sum of \eqref{eq: lsc lagrange summable form} from $n_k$ to $N>n_k$ with $(x,y)=(\overline{x},\overline{y})$ and obtain
\begin{align}
& 0\geq \sum_{n=n_k}^N \mathcal{L}\left(\overline{x},y_{n+1}\right)-\mathcal{L}\left(x_{n+1},\overline{y}\right) \nonumber\\
&\geq  \frac{1}{2\tau} \left[ \Vert \overline{y}-y_{N+1}\Vert^2 - \Vert \overline{y}-y_{n_k}\Vert^2\right] + \sum_{n=n_k}^N\frac{1-\sqrt{\sigma\tau}\Vert L\Vert}{2\tau}\Vert y_n-y_{n+1}\Vert^2  \nonumber\\
& +\left(\frac{1}{2\sigma}+a_{N+1}\right) \Vert \overline{x}-x_{N+1}\Vert^2 - \left(\frac{1}{2\sigma}+a_{n_k}\right) \Vert \overline{x}-x_{n_k}\Vert^2 \nonumber\\
& + \left(\frac{1}{2\sigma}+a_{N}\right)\Vert x_N -x_{N+1}\Vert^2 +\sum_{n=n_k}^{N-1} \left(\frac{1-\sqrt{\sigma \tau}\Vert L\Vert }{2\sigma}+a_{n-1} \right) \Vert x_n-x_{n-1}\Vert^2 \nonumber\\
& +\langle L(x_N -x_{N+1}), \overline{y}-y_{N+1}\rangle - \langle L(x_{n_k-1} -x_{n_k}), \overline{y}-y_{n_k}\rangle \nonumber.
\end{align}
As $1+2\sigma a_n >0$ for all $n\in \N$, we have
\begin{align}
& \frac{1}{2\tau} \Vert \overline{y}-y_{n_k}\Vert^2 +  \left(\frac{1}{2\sigma}+a_{n_k}\right) \Vert \overline{x}-x_{n_k}\Vert^2 \nonumber\\
& \geq \frac{1}{2\tau} \Vert \overline{y}-y_{N+1}\Vert^2 +
\left(\frac{1}{2\sigma}+a_{N+1}\right) \Vert \overline{x}-x_{N+1}\Vert^2 \nonumber\\
&+\langle L(x_N -x_{N+1}), \overline{y}-y_{N+1}\rangle - \langle L(x_{n_k-1} -x_{n_k}), \overline{y}-y_{n_k}\rangle. 
\label{eq: norm xnk and xN}
\end{align}
Letting $k\to \infty, N\to\infty$, the LHS and inner product terms of \eqref{eq: norm xnk and xN} go to zero, which implies $x_{N+1}\to \overline{x}$ and $y_{N+1}\to \overline{y}$. Hence, the whole sequences $(x_n)_{n\in\N}$ and $(y_n)_{n\in\N}$ converge to $\overline{x}$ and $\overline{y}$, respectively.
\end{proof}

\begin{remark}
Notice that, we use the splitting \eqref{eq: alg1 primal update subdiff form} instead of $x_{n+1}\in (J_\gamma+\partial_{lsc}^\mathbb{R} f)^{-1}J_\gamma(x_n- \gamma L^* y_{n+1})$  as  
    \[
    \left(a_{n}-a_{n+1}, (\frac{1}{\gamma}+2a_n)(x_n-\gamma L^* y_{n+1}) - (\frac{1}{\gamma}+2a_{n+1})x_{n+1} \right)\in \partial_{lsc}^\mathbb{R} f(x_{n+1}).
    \]
    The appearance of $L^* y_{n+1}$ does not allow us from proving the convergence results unless we assume that $(y_n)_{n\in\mathbb{N}}$ is bounded.
    On the other hand, we can achieve full splitting (i.e. with only $\partial_{lsc}^\mathbb{R} f(x_{n+1})$ on RHS of \eqref{eq: alg1 primal update subdiff form} by modifying $x_{n+1}\in (J_\gamma+\partial_{lsc}^\mathbb{R} f)^{-1}J_\gamma(x_n- \frac{\gamma}{1+2\gamma a_n} L^* y_{n+1})$. With this new update, we can achieve the same results as in Theorem \ref{thm: pd lsc prox convergence} with additional assumption $a_n>-1/2\sigma$ for all $n\in\mathbb{N}$.
\end{remark}

\begin{remark}
    The results in Theorem \ref{thm: pd lsc prox convergence}-\ref{thm: CP lsc-convex convergence-3} hold for a large class of functions which are $\Phi_{lsc}^\mathbb{R}$-convex under mild assumptions. One important reason of this fact is the assumption $(a_n)_{n\in\N}$ converges in order to obtain convergence of the constructed sequences to a saddle point. 
\end{remark}

When $f$ is a $\rho$-weakly convex function (i.e. $f+\frac{\rho}{2}\Vert \cdot\Vert^2$ is convex), we can obtain ergodic convergence rate. The following property of weakly convex function will be useful. 
\begin{proposition}
\label{prop: jensen wc}
Let $f:X\to (-\infty,+\infty]$ be a $\rho$-weakly convex function. Then for $i=1,n, x_i\in X$, $\lambda_i \in [0,1], \sum_{i=1}^n \lambda_i =1$, it holds
\begin{equation}
\label{eq: jensen inequality weakly convex}
f\left(\sum_{i=1}^{n}\lambda_{i}x_{i}\right)\leq\sum_{i=1}^{n}\lambda_{i}f\left(x_{i}\right)+ \frac{\rho}{2}\sum_{i>j}^{n}\lambda_{i}\lambda_{j}\left\Vert x_{i}-x_{j}\right\Vert ^{2}
\end{equation}
\end{proposition}
\begin{proof}
    This can be proved using induction combine with definition of weakly convex function.
\end{proof}
We state our result.
\begin{theorem}
Let $f:X\to (-\infty,+\infty]$ be a proper lsc $\rho$-weakly convex, $g:Y\to (-\infty,+\infty]$ be a proper lsc convex function and $L:X\to Y$ be a bounded linear operator. Let $(x_n)_{n\in\N},(y_n)_{n\in\N}$ and $(a_n)_{n\in\N}$ be the sequences generated by Algorithm \ref{alg: lsc Chambolle-Pock}. Assume that the set of saddle points $S$ is nonempty and $\tau\sigma \Vert L\Vert^2 <1$ and $1+2\sigma a_n > \sqrt{\sigma \tau}\| L \|$ for all $n\in \N$. For any $(x,y)\in X\times Y$ and any $N>1$, we have 
\begin{equation}
\label{eq: ergodic rate wc Lagrange}
\mathcal{L}\left(x,\bar{y}_{N}\right)-\mathcal{L}\left(\bar{x}_{N},y\right) \leq \frac{1}{N+1}\left[ \frac{1}{2\tau } \Vert y-y_{0}\Vert^2 + \left(\frac{1}{2\sigma}+a_{0}\right) \Vert x-x_0\Vert^2\right] + \frac{2\rho N}{N+1} C_S^2,
\end{equation}
where $\bar{x}_N = \frac{1}{N+1}\sum_{i=0}^N x_i, \bar{y}_N = \frac{1}{N+1}\sum_{i=0}^N y_i$ and for some constant $C_S \geq0$.
\end{theorem}

\begin{proof}
From Theorem \ref{thm: pd lsc prox convergence}, for all $(x,y)\in X\times Y$, we have

\begin{align}
\sum_{n=0}^N \mathcal{L}\left(x,y_{n+1}\right)-\mathcal{L}\left(x_{n+1},y\right) 
&\geq  \frac{1-\sqrt{\sigma\tau}\Vert L\Vert}{2\tau} \Vert y-y_{N+1}\Vert^2 - \frac{1}{2\tau } \Vert y-y_{0}\Vert^2 \nonumber\\
& +\left(\frac{1}{2\sigma}+a_{N+1}\right) \Vert x-x_{N+1}\Vert^2 - \left(\frac{1}{2\sigma}+a_{0}\right) \Vert x-x_0\Vert^2.
\end{align}
Let $\bar{x}_N = \frac{1}{N+1}\sum_{i=0}^N x_i, \bar{y}_N = \frac{1}{N+1}\sum_{i=0}^N y_i$ and by applying \eqref{eq: jensen inequality weakly convex} from Proposition \ref{prop: jensen wc} to the function $\mathcal{L}(\cdot,y)$. We arrive at
\begin{align}
\mathcal{L}\left(\bar{x}_{N},y\right)- \mathcal{L}\left(x,\bar{y}_{N}\right) -\frac{\rho}{2 (N+1)^2}\sum_{i,j=0;i>j}^{N} \left\Vert x_{i}-x_{j}\right\Vert ^{2}
& \leq \frac{1}{N+1} \left[\sum_{n=0}^N \mathcal{L}\left(x_{n+1},y\right) - \mathcal{L}\left(x,y_{n+1}\right)\right] \nonumber\\
& \leq \frac{1}{N+1}\left[ \frac{1}{2\tau } \Vert y-y_{0}\Vert^2 + \left(\frac{1}{2\sigma}+a_{0}\right) \Vert x-x_0\Vert^2\right].
\label{eq: ergo wc lagrange}
\end{align}
Since $(x_n)_{n\in\mathbb{N}}$ and $(y_n)_{n\in\mathbb{N}}$ are bounded, for every $(x^*,y^*)\in S$, there exists $C(x^*)>0$ such that $\Vert x_n -x^*\Vert \leq C(x^*)$ for all $n\in\mathbb{N}$. By using Cauchy-Schwarz's inequality on \eqref{eq: ergo wc lagrange}, we obtain
\begin{align*}
\mathcal{L}\left(\bar{x}_{N},y\right)- \mathcal{L}\left(x,\bar{y}_{N}\right) 
& \leq \frac{1}{N+1}\left[ \frac{1}{2\tau } \Vert y-y_{0}\Vert^2 + \left(\frac{1}{2\sigma}+a_{0}\right) \Vert x-x_0\Vert^2\right] + \frac{\rho}{2 (N+1)^2}\sum_{i,j=0;i>j}^{N} \left\Vert x_{i}-x_{j}\right\Vert ^{2}
\nonumber\\
& \leq \frac{1}{N+1}\left[ \frac{1}{2\tau } \Vert y-y_{0}\Vert^2 + \left(\frac{1}{2\sigma}+a_{0}\right) \Vert x-x_0\Vert^2\right] \nonumber\\
& + \frac{\rho}{(N+1)^2}\sum_{i,j=0;i>j}^{N} \left\Vert x_{i}- x^* \Vert^2 + \Vert x_{j} - x^*\right\Vert ^{2} \nonumber\\
& \leq \frac{1}{N+1}\left[ \frac{1}{2\tau } \Vert y-y_{0}\Vert^2 + \left(\frac{1}{2\sigma}+a_{0}\right) \Vert x-x_0\Vert^2\right] + \frac{\rho}{(N+1)^2}\sum_{i,j=0;i>j}^{N} 2C^2(x^*) \nonumber\\
& =  \frac{1}{N+1}\left[ \frac{1}{2\tau } \Vert y-y_{0}\Vert^2 + \left(\frac{1}{2\sigma}+a_{0}\right) \Vert x-x_0\Vert^2\right] + \frac{2\rho N(N+1)}{(N+1)^2} C^2 (x^*).
\end{align*}
As $(x^*,y^*)\in S$ can be taken arbitrary, we can choose the smallest $C_S = C(x^*)$ which gives us \eqref{eq: ergodic rate wc Lagrange}.
\end{proof}

Notice that \eqref{eq: ergodic rate wc Lagrange} give us a worse rate compare to the convex case which is $1/N$, due to the last term on the RHS. Nonetheless, our setting applies for $\Phi_{lsc}^\mathbb{R}$-convex $f$ and convex $g$. 
Moreover, Theorem \ref{thm: pd lsc prox convergence} does not rely on assumption such as K{\L}, or error bound condition. On the other hand, the rate of convergence can be improved e.g. such as $\log N/N$ as in \cite{shumaylov2024weakly}, using K{\L} condition.

\section{Fully \texorpdfstring{$\Phi_{lsc}^\mathbb{R}$}\ -\ Setting}
\label{sec: full lsc problem}
In this section, we consider $g:X\to (-\infty,+\infty], L=\mathrm{Id}$ and replace the convexity of $g$ by $\Phi_{lsc}^\mathbb{R}$-convexity. The saddle point problem we aim to solve has the form,
\begin{equation}
\label{prob: DCP L=Id}
\inf_{x\in X} \sup_{\phi \in \Phi_{lsc}^\mathbb{R}} f(x)+\phi (x)-g^*_\Phi (\phi).
\end{equation}
Problem \eqref{prob: DCP L=Id} is a generalized version of \eqref{prob: Lagrange g convex}. The optimality condition for problem \eqref{prob: DCP L=Id} is defined: $(\hat{x},\hat{\phi})\in X\times \Phi_{lsc}^\mathbb{R}$ is a saddle point if and only if 
\begin{equation}
\label{eq: kkt full lsc L=Id}
-\hat{\phi} \in \partial_{lsc}^\mathbb{R} f(\hat{x}), \qquad \hat{x} \in \partial_{X} g^*_\Phi(\hat{\phi}).
\end{equation}

\subsection{Applying convex subdifferentials to the conjugate}
We notice that $g^*_\Phi (\phi)$ is a convex function with respect to $\phi\in\Phi_{lsc}^\mathbb{R}$. We can use convex subdifferentials and convex proximal operator to $g^*_\Phi$ when $\Phi_{lsc}^\mathbb{R}$ is equipped with appropriated norm and inner product. 
Observe that we can identify $\Phi_{lsc}^\mathbb{R}$ with $\mathbb{R}\times X$. For any $\phi\in\Phi_{lsc}^\mathbb{R}$, we have $\phi \equiv (a,u)$ which means $\phi(\cdot) = -a\Vert \cdot\Vert^2 + \langle u,\cdot\rangle$.
Consequently, $\Phi_{lsc}^\mathbb{R}$ inherits the structure of the space $\mathbb{R}\times X$, so we define the inner product on it as 
    \begin{align}
    \langle \cdot,\cdot\rangle_\Phi & :\Phi_{lsc}^\mathbb{R} \times \Phi_{lsc}^\mathbb{R} \to \mathbb{R} \nonumber\\
    \langle \phi,\psi\rangle_\Phi & = \langle (a_\phi,u_\phi),(a_\psi,u_\psi)\rangle_\Phi = a_\phi a_\psi + \langle u_\phi,u_\psi\rangle_X, 
    \label{eq: lsc inner prod}
    \end{align}
and the norm induced by the inner product $\Vert \phi \Vert_\Phi = \sqrt{\langle \phi,\phi\rangle_\Phi}$. 
The pair $(\bar{x},\bar{\phi})$ is a saddle point of problem \eqref{prob: DCP L=Id} if and only if it satisfies the KKT condition (which is derived from \eqref{eq: kkt full lsc L=Id})
\begin{align}
-\bar{\phi}\in \partial_{lsc}^\mathbb{R} f(\bar{x}), \qquad (-\Vert \bar{x}\Vert^2,\bar{x}) \in \partial g^*_\Phi (\bar{\phi}),
\end{align}
where $ \partial g^*_\Phi$ is the convex subdifferentials defined on the space $\mathbb{R}\times X$.
With this, we propose primal-dual algorithm for this case

\begin{algorithm}[H]\caption{Fully $\lsc^\R$-Chambolle-Pock Algorithm}\label{alg: lsc Chambolle-Pock full L=Id}
\textbf{Initialize:} Choose $\tau,\sigma >0, \sigma\tau<1$ and $\bar{a}_0>(\sqrt{\sigma\tau}-1)/2\sigma$. Starting $(x_0,\phi_0)\in X\times \Phi_{lsc}^\mathbb{R}$ and $x_0=x_{-1}$.\\
{\normalfont\textbf{Update: }}{
For $n\in\N$,
\begin{itemize}
    \item Dual step update:
    \begin{itemize}
        \item[$\blacksquare$] $\phi_{n+1}  =\argmind{\phi\in \Phi_{lsc}^\mathbb{R}}{ g^*_\Phi(y)- \langle \phi, (-2\Vert x_n\Vert^2+\Vert x_{n-1}\Vert^2, 2x_n-x_{n-1})\rangle_\Phi +\frac{1}{2\tau} \Vert \phi - \phi_n \Vert^2_\Phi}$
    \end{itemize}
    \item Primal step update:
\begin{itemize}
        \item[$\blacksquare$] Pick $\left( \bar{a}_n, (\frac{1}{\sigma}+2\bar{a}_n)x_n\right) \in J_\sigma (x_n)$ according to \eqref{eq: J_gamma form}
        \item[$\blacksquare$] Pick $x_{n+1} \in \argmind{z\in X}{ f(z) + \phi_{n+1} (z)+\left( \frac{1}{2\sigma} +\bar{a}_n\right) \Vert z-x_n\Vert^2}$
    \end{itemize}
    \item \textbf{If} (there is no such $\bar{a}_n$.) \textbf{Stop the algorithm}.
\end{itemize}}
\end{algorithm}
\vspace{0.5cm}

Notice that the dual step update can be seen as a proximal update of the form
\[ \phi_{n+1} = \mathrm{prox}_{\tau g^*_\Phi} (\phi_n -\tau (2\bar{x}_n-\bar{x}_{n-1})),\]
where $\bar{x}_n = (-\Vert x_n\Vert^2,x_n)\in \mathbb{R}\times X$, which coincides with the one in Chambolle-Pock Algorithm \cite{chambolle2011first}. 
Similar to Algorithm \ref{alg: lsc Chambolle-Pock}, we obtain the following relations for Algorithm \ref{alg: lsc Chambolle-Pock full L=Id},
\begin{align}
\frac{\phi_n -\phi_{n+1}}{\tau} +\left( -2\Vert x_n\Vert^2 +\Vert x_{n-1}\Vert^2, 2x_n-x_{n-1}\right) & \in \partial g^*_\Phi (\phi_{n+1}) 
\label{eq: alg2 update g*}\\
\left(\bar{a}_n-\bar{a}_{n+1}, \left(\frac{1}{\sigma} +2\bar{a}_n\right) x_n - \left(\frac{1}{\sigma} +2\bar{a}_{n+1}\right) x_{n+1}\right) -\phi_{n+1} & \in \partial_{lsc}^\mathbb{R} f(x_{n+1}).
\label{eq: alg2 update f}
\end{align}

\begin{theorem}
\label{thm: pd full lsc prox convergence}
Let $f,g:X\to (-\infty,+\infty]$ be a proper $\lsc^\R$-convex. Let $(x_n)_{n\in\N},(\phi_n)_{n\in\N}=(a_n,u_n)_{n\in\mathbb{N}}, (\bar{x}_{n})_{n\in\N}$ and $(\bar{a}_n)_{n\in\N}$ be the sequences generated by Algorithm \ref{alg: lsc Chambolle-Pock full L=Id}. Assume that $\tau\sigma <1$ and $1+2\sigma \bar{a}_n > \sqrt{\sigma \tau}$ for all $n\in \N$, then we have
\begin{enumerate}[label=\roman*.]
    \item For any $(x,\phi)\in X\times \Phi_{lsc}^\mathbb{R}$ and any $N\in\mathbb{N}$, we have
    \begin{align}
 & \sum_{i=0}^{N}\mathcal{L}\left(x,\phi_{i+1}\right)-\mathcal{L}\left(x_{i+1},\phi\right) \nonumber\\
\geq & \left(\frac{1}{2\sigma}+\bar{a}_{N+1}\right)\left\Vert x-x_{N+1}\right\Vert ^{2}-\left(\frac{1}{2\sigma}+\bar{a}_{0}\right)\left\Vert x-x_{0}\right\Vert ^{2} -\frac{1}{2\tau}\left\Vert \phi-\phi_{0}\right\Vert ^{2}_\Phi \nonumber\\
 &+ \frac{\sqrt{\sigma\tau}}{2}\left( \frac{a-a_{N+1}}{\sqrt{\tau}} + \frac{\left\Vert x_{N+1}\right\Vert ^{2}-\left\Vert x_{N}\right\Vert ^{2}}{\sqrt{\sigma}} \right)^{2} +\frac{1-\sqrt{\sigma\tau}}{2\tau} \Vert \phi-\phi_{N+1}\Vert^2_\Phi \nonumber\\
 & +\sum_{i=0}^{N} { \left(1-\sqrt{\sigma\tau}+2\sigma\bar{a}_{i} - \sqrt{\sigma\tau}(\Vert x_i\Vert+\Vert x_{i+1}\Vert)^2 \right) \frac{\left\Vert x_{i}-x_{i+1}\right\Vert ^{2}}{2\sigma}} \nonumber\\
 & +\frac{\sqrt{\sigma\tau}}{2}\left( \frac{a_i-a_{i+1}}{\sqrt{\tau}} + \frac{\left\Vert x_{i-1}\right\Vert ^{2}-\left\Vert x_{i}\right\Vert ^{2}}{\sqrt{\sigma}}\right)^{2} +\frac{1-\sqrt{\sigma\tau}}{2\tau} \Vert \phi_i-\phi_{i+1}\Vert^2_\Phi. 
\label{eq: full lsc lagrange summable final form}
\end{align}
    where $\mathcal{L} (x,\phi) = f(x)+ \phi(x) -g^*_\Phi (\phi)$ is the Lagrangian.
    \item\label{thm: full lsc convergence-2} If the sequences $(x_n)_{n\in\mathbb{N}}$ and $(\bar{a}_n)_{n\in\mathbb{N}}$ are such that
    \begin{equation}
    \label{eq: alg2 ass boundedness}
    \frac{1+2\sigma\bar{a}_{n}}{\sqrt{\sigma\tau}} \geq 1+(\Vert x_n\Vert+\Vert x_{n+1}\Vert)^2 ,
\end{equation}
for all $n\in\mathbb{N}$. Then the sequences $(x_n,\phi_n)_{n\in\N}$ are bounded.
    \item\label{thm: full lsc convergence-3} If $X$ is finite dimensional, we have \eqref{eq: alg2 ass boundedness} holds and $\bar{a}_n-\bar{a}_{n+1}\to 0$. Then $(x_n,\phi_n)_{n\in\N}$ converges to a KKT point in the sense of \eqref{eq: kkt full lsc L=Id}.
    
\end{enumerate}
\end{theorem}
\begin{proof}
From Algorithm \eqref{alg: lsc Chambolle-Pock full L=Id} and \eqref{eq: alg2 update g*}, \eqref{eq: alg2 update f}, we have the estimation in term of Lagrange function 
\begin{align}
& \mathcal{L}\left(x,\phi_{n+1}\right)-\mathcal{L}\left(x_{n+1},\phi\right) \nonumber\\
& \geq\left(\frac{1}{2\sigma}+\bar{a}_{n+1}\right)\left\Vert x-x_{n+1}\right\Vert ^{2}+\left(\frac{1}{2\sigma}+\bar{a}_{n}\right)\left(\left\Vert x_{n}-x_{n+1}\right\Vert ^{2}-\left\Vert x-x_{n}\right\Vert ^{2}\right)\nonumber\\
 & +\frac{1}{2\tau}\left(\left\Vert \phi-\phi_{n+1}\right\Vert _{\Phi}^{2}-\left\Vert \phi-\phi_{n}\right\Vert _{\Phi}^{2}+\left\Vert \phi_{n}-\phi_{n+1}\right\Vert _{\Phi}^{2}\right) \nonumber\\
 & +\phi_{n+1}\left(x\right)-\phi\left(x_{n+1}\right) + \phi_{n+1} (x_{n+1})-\phi_{n+1} (x) \nonumber\\
 & +2\phi\left(x_{n}\right)-\phi\left(x_{n-1}\right)-2\phi_{n+1}\left(x_{n}\right)+\phi_{n+1}\left(x_{n-1}\right).
 \label{eq: PD full lsc lagrange estimate}
\end{align}

The last two lines of the above estimate can be written as
\begin{align}
& \left[\left(\phi_{n+1}-\phi\right)\left(x_{n+1}\right)-\left(\phi_{n+1}-\phi\right)\left(x_{n}\right)\right]-\left[\left(\phi_{n}-\phi\right)\left(x_{n}\right)-\left(\phi_{n}-\phi\right)\left(x_{n-1}\right)\right] \nonumber\\
& +\left[\left(\phi_{n+1}-\phi_{n}\right)\left(x_{n-1}\right)-\left(\phi_{n+1}-\phi_{n}\right)\left(x_{n}\right)\right].
\label{eq: PD full lsc auxiliary terms}
\end{align}
We analyze the first bracket of \eqref{eq: PD full lsc auxiliary terms}
\begin{align}
& \left[\left(\phi_{n+1}-\phi\right)\left(x_{n+1}\right)-\left(\phi_{n+1}-\phi\right)\left(x_{n}\right)\right]	\nonumber\\
&=\left(a-a_{n+1}\right)\left(\left\Vert x_{n+1}\right\Vert ^{2}-\left\Vert x_{n}\right\Vert ^{2}\right)+\left\langle u_{n+1}-u,x_{n+1}-x_{n}\right\rangle \nonumber\\
&\geq\left(a-a_{n+1}\right)\left(\left\Vert x_{n+1}\right\Vert ^{2}-\left\Vert x_{n}\right\Vert ^{2}\right)-\frac{\sqrt{\sigma\tau}}{2\tau}\left\Vert u-u_{n+1}\right\Vert ^{2}-\frac{\sqrt{\sigma\tau}}{2\sigma}\left\Vert x_{n+1}-x_{n}\right\Vert ^{2}.
\label{eq: PD full lsc aux term estimate}
\end{align}

Let us take the sum of \eqref{eq: PD full lsc lagrange estimate} from $0$ till $N$, reminding $x_{-1} =x_0$ and using \eqref{eq: PD full lsc aux term estimate}, we obtain
\begin{align}
 & \sum_{i=0}^{N}\mathcal{L}\left(x,\phi_{i+1}\right)-\mathcal{L}\left(x_{i+1},\phi\right) \nonumber\\
\geq & \left(\frac{1}{2\sigma}+\bar{a}_{N+1}\right)\left\Vert x-x_{N+1}\right\Vert ^{2}-\left(\frac{1}{2\sigma}+\bar{a}_{0}\right)\left\Vert x-x_{0}\right\Vert ^{2} +\frac{1}{2\tau}\left(\left\Vert \phi-\phi_{N+1}\right\Vert ^{2}_\Phi -\left\Vert \phi-\phi_{0}\right\Vert ^{2}_\Phi\right) \nonumber\\
 &+ \left(a-a_{N+1}\right)\left(\left\Vert x_{N+1}\right\Vert ^{2}-\left\Vert x_{N}\right\Vert ^{2}\right)-\frac{\sqrt{\sigma\tau}}{2\tau}\left\Vert u-u_{N+1}\right\Vert ^{2}-\frac{\sqrt{\sigma\tau}}{2\sigma}\left\Vert x_{N+1}-x_{N}\right\Vert ^{2} \nonumber\\
 & +\sum_{i=0}^{N}\left(\frac{1}{2\sigma}+\bar{a}_{i}\right)\left\Vert x_{i}-x_{i+1}\right\Vert ^{2}+\frac{1}{2\tau}\left\Vert \phi_{i}-\phi_{i+1}\right\Vert ^{2}_\Phi + \left[\left(\phi_{i+1}-\phi_{i}\right)\left(x_{i-1}\right)-\left(\phi_{i+1}-\phi_{i}\right)\left(x_{i}\right)\right].
 \label{eq: PD full lsc lagrange sum}
\end{align}
We expand the term $\Vert \phi -\phi_{N+1}\Vert_{\Phi}^2 = (a-a_{N+1})^2 +\Vert u-u_{N+1}\Vert^2$ and simplify with the third line of \eqref{eq: PD full lsc lagrange sum}.
\begin{align}
 & \Vert \phi -\phi_{N+1}\Vert_{\Phi}^2 + \left(a-a_{N+1}\right)\left(\left\Vert x_{N+1}\right\Vert ^{2}-\left\Vert x_{N}\right\Vert ^{2}\right)-\frac{\sqrt{\sigma\tau}}{2\tau}\left\Vert u-u_{N+1}\right\Vert ^{2}-\frac{\sqrt{\sigma\tau}}{2\sigma}\left\Vert x_{N+1}-x_{N}\right\Vert ^{2} \nonumber\\
 & = \frac{\sqrt{\sigma\tau}}{2}\left( \frac{a-a_{N+1}}{\sqrt{\tau}} + \frac{\left\Vert x_{N+1}\right\Vert ^{2}-\left\Vert x_{N}\right\Vert ^{2}}{\sqrt{\sigma}} \right)^{2}-\frac{\sqrt{\sigma\tau}}{2\sigma}\left(\left\Vert x_{N+1}\right\Vert ^{2}-\left\Vert x_{N}\right\Vert ^{2}\right)^{2} \nonumber\\
 & +\frac{1-\sqrt{\sigma\tau}}{2\tau} \Vert \phi-\phi_{N+1}\Vert^2_\Phi -\frac{\sqrt{\sigma\tau}}{2\sigma}\left\Vert x_{N+1}-x_{N}\right\Vert ^{2}
\label{eq: PD full lsc phi-phiN+1}
\end{align}
Consider 
\[
\left(\left\Vert x_{N+1}\right\Vert ^{2}-\left\Vert x_{N}\right\Vert ^{2}\right)^{2} = (\left\Vert x_{N+1}\right\Vert +\left\Vert x_{N}\right\Vert)^2 (\left\Vert x_{N+1}\right\Vert -\left\Vert x_{N}\right\Vert )^2 \leq 
(\left\Vert x_{N+1}\right\Vert +\left\Vert x_{N}\right\Vert)^2  \Vert x_{N+1}-x_N\Vert^2.
\]
Doing the same for $\left(\phi_{i+1}-\phi_{i}\right)\left(x_{i-1}\right)-\left(\phi_{i+1}-\phi_{i}\right)\left(x_{i}\right)$ in the summation, we obtain
\begin{align}
 & \sum_{i=0}^{N}\mathcal{L}\left(x,\phi_{i+1}\right)-\mathcal{L}\left(x_{i+1},\phi\right) \nonumber\\
\geq & \left(\frac{1}{2\sigma}+\bar{a}_{N+1}\right)\left\Vert x-x_{N+1}\right\Vert ^{2}-\left(\frac{1}{2\sigma}+\bar{a}_{0}\right)\left\Vert x-x_{0}\right\Vert ^{2} -\frac{1}{2\tau}\left\Vert \phi-\phi_{0}\right\Vert ^{2}_\Phi \nonumber\\
 &+ \frac{\sqrt{\sigma\tau}}{2}\left( \frac{a-a_{N+1}}{\sqrt{\tau}} + \frac{\left\Vert x_{N+1}\right\Vert ^{2}-\left\Vert x_{N}\right\Vert ^{2}}{\sqrt{\sigma}} \right)^{2} +\frac{1-\sqrt{\sigma\tau}}{2\tau} \Vert \phi-\phi_{N+1}\Vert^2_\Phi \nonumber\\
 & +\sum_{i=0}^{N} { \left(1-\sqrt{\sigma\tau}+2\sigma\bar{a}_{i} - \sqrt{\sigma\tau}(\Vert x_i\Vert+\Vert x_{i+1}\Vert)^2 \right) \frac{\left\Vert x_{i}-x_{i+1}\right\Vert ^{2}}{2\sigma}} \nonumber\\
 & +\frac{\sqrt{\sigma\tau}}{2}\left( \frac{a_i-a_{i+1}}{\sqrt{\tau}} + \frac{\left\Vert x_{i-1}\right\Vert ^{2}-\left\Vert x_{i}\right\Vert ^{2}}{\sqrt{\sigma}}\right)^{2} +\frac{1-\sqrt{\sigma\tau}}{2\tau} \Vert \phi_i-\phi_{i+1}\Vert^2_\Phi. 
\label{eq: full lsc lagrange summable final form proof-1}
\end{align}
We obtain \eqref{eq: full lsc lagrange summable final form} for all $(x,\phi)\in X\times \Phi_{lsc}^\mathbb{R}$.

For the second statement, letting $(x,\phi)=(x^*,\phi^*)$ be a saddle point in \eqref{eq: full lsc lagrange summable final form}, we have
\begin{align}
& 0 \geq \sum_{i=0}^{N}\mathcal{L}\left(x^*,\phi_{i+1}\right)-\mathcal{L}\left(x_{i+1},\phi^*\right) \nonumber\\
\geq & \left(\frac{1}{2\sigma}+\bar{a}_{N+1}\right)\left\Vert x^*-x_{N+1}\right\Vert ^{2}-\left(\frac{1}{2\sigma}+\bar{a}_{0}\right)\left\Vert x^*-x_{0}\right\Vert ^{2} -\frac{1}{2\tau}\left\Vert \phi^*-\phi_{0}\right\Vert ^{2}_\Phi \nonumber\\
 &+ \frac{\sqrt{\sigma\tau}}{2}\left( \frac{a^*-a_{N+1}}{\sqrt{\tau}} + \frac{\left\Vert x_{N+1}\right\Vert ^{2}-\left\Vert x_{N}\right\Vert ^{2}}{\sqrt{\sigma}} \right)^{2} +\frac{1-\sqrt{\sigma\tau}}{2\tau} \Vert \phi^*-\phi_{N+1}\Vert^2_\Phi \nonumber\\
 & +\sum_{i=0}^{N} { \left(1-\sqrt{\sigma\tau}+2\sigma\bar{a}_{i} - \sqrt{\sigma\tau}(\Vert x_i\Vert+\Vert x_{i+1}\Vert)^2 \right) \frac{\left\Vert x_{i}-x_{i+1}\right\Vert ^{2}}{2\sigma}}\nonumber\\
 & +\frac{\sqrt{\sigma\tau}}{2}\left( \frac{a_i-a_{i+1}}{\sqrt{\tau}} + \frac{\left\Vert x_{i-1}\right\Vert ^{2}-\left\Vert x_{i}\right\Vert ^{2}}{\sqrt{\sigma}}\right)^{2} +\frac{1-\sqrt{\sigma\tau}}{2\tau} \Vert \phi_i-\phi_{i+1}\Vert^2_\Phi. 
\nonumber
\end{align}
Thanks to assumption \eqref{eq: alg2 ass boundedness}, we can remove the summation to obtain
\begin{align*}
& \left(\frac{1}{2\sigma}+\bar{a}_{0}\right)\left\Vert x^*-x_{0}\right\Vert ^{2} +\frac{1}{2\tau}\left\Vert \phi^*-\phi_{0}\right\Vert ^{2}_\Phi \\
\geq &\left(\frac{1}{2\sigma}+\bar{a}_{N+1}\right)\left\Vert x^*-x_{N+1}\right\Vert ^{2} 
+\frac{1-\sqrt{\sigma\tau}}{2\tau} \Vert \phi^*-\phi_{N+1}\Vert^2_\Phi
\end{align*}
The above inequality holds for any $N\in \mathbb{N}$, which implies that $(x_n,\phi_n)_{n\in\mathbb{N}}$ are bounded.

For the convergence of $(x_n,\phi_n)_{n\in\mathbb{N}}$, the proof has the same argument as of Theorem \ref{thm: pd lsc prox convergence}. We just need to prove that 
\[
x_n-x_{n+1} \to 0, \quad \phi_n -\phi_{n+1} \to 0.
\]
From \eqref{eq: full lsc lagrange summable final form} with $(x,\phi)=(x^*,\phi^*)$ being a saddle point, we infer
\begin{align}
& \left(\frac{1}{2\sigma}+\bar{a}_{0}\right)\left\Vert x^*-x_{0}\right\Vert ^{2} -\frac{1}{2\tau}\left\Vert \phi^*-\phi_{0}\right\Vert ^{2}_\Phi \nonumber\\
 & \geq\sum_{i=0}^{N} { \left(1-\sqrt{\sigma\tau}+2\sigma\bar{a}_{i} - \sqrt{\sigma\tau}(\Vert x_i\Vert+\Vert x_{i+1}\Vert)^2 \right) \frac{\left\Vert x_{i}-x_{i+1}\right\Vert ^{2}}{2\sigma}} + \nonumber\\
 & +\frac{\sqrt{\sigma\tau}}{2}\left( \frac{a_i-a_{i+1}}{\sqrt{\tau}} + \frac{\left\Vert x_{i-1}\right\Vert ^{2}-\left\Vert x_{i}\right\Vert ^{2}}{\sqrt{\sigma}}\right)^{2} +\frac{1-\sqrt{\sigma\tau}}{2\tau} \Vert \phi_i-\phi_{i+1}\Vert^2_\Phi.  \nonumber
\end{align}
This holds for any $N\in\mathbb{N}$, we can let $N\to \infty$ to obtain 
\begin{align}
\label{eq: alg2 proof summable x}
& { \left(1-\sqrt{\sigma\tau}+2\sigma\bar{a}_{i} - \sqrt{\sigma\tau}(\Vert x_i\Vert+\Vert x_{i+1}\Vert)^2 \right) \frac{\left\Vert x_{i}-x_{i+1}\right\Vert ^{2}}{2\sigma}} \to 0,
 \\
\label{eq: alg2 proof summable x a}
& \left( \frac{a_i-a_{i+1}}{\sqrt{\tau}} + \frac{\left\Vert x_{i-1}\right\Vert ^{2}-\left\Vert x_{i}\right\Vert ^{2}}{\sqrt{\sigma}}\right)^{2}  \to 0, \\
& \Vert \phi_i-\phi_{i+1}\Vert^2_\Phi \to 0.
\label{eq: alg2 proof summable u}
\end{align}
We have $\phi_i-\phi_{i+1} \to 0$ from \eqref{eq: alg2 proof summable u},
and $x_i-x_{i+1}\to 0$ from the assumptions and \eqref{eq: alg2 proof summable x a}, \eqref{eq: alg2 proof summable x}. The rest of the proof is analogous to the one in Theorem \ref{thm: pd lsc prox convergence}.
\end{proof}

\begin{remark}
When $g$ is convex, we can consider $g^*_\Phi$ as $g^*$ and Algorithm \ref{alg: lsc Chambolle-Pock full L=Id} is exactly Algorithm \ref{alg: lsc Chambolle-Pock} and the assumption \eqref{eq: alg2 ass boundedness} can be relaxed to obtain the convergence.
The appearance of the forth-order terms in \eqref{eq: PD full lsc phi-phiN+1} and the additional assumption \eqref{eq: alg2 ass boundedness} come from $\phi_n$. For the case $g$ is convex, the assumption \eqref{eq: alg2 ass boundedness} is not needed for the boundedness of the iterates as in Theorem \ref{thm: pd lsc prox convergence}.
\end{remark}

\begin{remark}
    One important aspect of convergence of Algorithm \ref{alg: lsc Chambolle-Pock full L=Id} is that we need to control the sequence $(\bar{a}_n)_{n\in\mathbb{N}}$. Compare to Theorem \ref{thm: pd lsc prox convergence}, we need to control both $(x_n)_{n\in\mathbb{N}}$ and $(\bar{a}_n)_{n\in\mathbb{N}}$ for \eqref{eq: alg2 ass boundedness} to hold. Notice the RHS of Assumption \eqref{eq: alg2 ass boundedness} is of order forth, so it can be satisfied if $x_{n+1}$ is not too far from $x_n$. We can add it as a restriction to Algorithm \ref{alg: lsc Chambolle-Pock full L=Id}. Overall, the assumptions of Theorem \ref{thm: pd full lsc prox convergence}-(\ref{thm: full lsc convergence-2},\ref{thm: full lsc convergence-3}) are the trade-offs between a bi-linear and nonlinear, nonconvex coupling term $\phi(x)$.
\end{remark}

\begin{remark}
    The structure of $\Phi_{lsc}^\mathbb{R}$ makes the coupling term $\phi(x)$ of the Lagrangian $\mathcal{L}(x,\phi)$ in \eqref{prob: DCP L=Id} nonlinear in $x\in X$ and linear in $\phi =(a,u)\in \Phi_{lsc}^\mathbb{R}$. It may or may not be nonconvex in $x$ and also locally and not globally Lipschitz continuous in $x$. 
    Therefore, problem \eqref{prob: DCP L=Id} is more general than the one considered in \cite{hamedani2021primal} as well as in \cite{clason2021primal}. Moreover, we obtain convergence in Theorem \ref{thm: pd full lsc prox convergence} by restricting the sequences $(a_n)_{n\in\mathbb{N}}$ and $(\bar{a}_n)_{n\in\mathbb{N}}$ (assumptions \eqref{eq: alg2 ass boundedness} and \ref{thm: full lsc convergence-3}) without using additional conditions like H\"older error bound \cite{johnstone2020faster,cuong2022error} or K\L-inequality \cite{attouch2013convergence}.
\end{remark}
\subsection{Using \texorpdfstring{$\Phi_{lsc}^\mathbb{R} $}--Subdifferentials for conjugate function}
In this subsection, we attempt to design primal-dual algorithm using abstract subdifferentials \eqref{eq: lsc subgrad conj} on the conjugate function $g^*_\Phi(\phi)$.
Since we consider $\phi$ as a variable, we can write $\phi (x) = x(\phi)$ as a function of $\phi\in \Phi_{lsc}^\mathbb{R}$. This notation will be used for presentation convenience below.

Using Definition \ref{def:lsc-subdiff}, we propose the following algorithm


\begin{algorithm}[H]\caption{$\lsc^\R$-Chambolle-Pock Algorithm}\label{alg: lsc Chambolle-Pock full L=Id v2}
\textbf{Initialize:} Choose $\tau,\sigma >0, \sigma\tau<1$ and $\bar{a}_0>(\sqrt{\sigma\tau}-1)/2\sigma$. Starting $(x_0,\phi_0)\in X\times \Phi_{lsc}^\mathbb{R}$ and $x_0=x_{-1}$.\\
{\normalfont\textbf{Update: }}{
For $n\in\N$,
\begin{itemize}
    \item Find $\phi_{n+1}=(a_{n+1},u_{n+1})$ such that 
    \begin{align}
    \label{eq: proximal update conjugate v2}
    \frac{u_n -u_{n+1}}{\gamma} &\in \partial_{\Phi}^\mathbb{R} (g^*_\Phi - \langle \cdot, (-2\Vert x_n\Vert^2+\Vert x_{n-1}\Vert^2, 2x_n-x_{n-1})\rangle_\Phi) (\phi_{n+1}), \\
    \left\Vert u_{n}-u_{n+1}\right\Vert ^{2} &\leq a_{n+1}-a_{n}.\nonumber
    \end{align}
    \item \textbf{If} (there is no such $\phi_{n+1}$) \textbf{or} ($\bar{a}_n<-1/2\sigma$), \textbf{Stop the algorithm}.
    \item We update the primal step as follow:
\begin{itemize}
        \item[$\blacksquare$] Pick $\left( \bar{a}_n, (\frac{1}{\sigma}+2\bar{a}_n)x_n\right) \in J_\sigma (x_n)$ according to \eqref{eq: J_gamma form}
        \item[$\blacksquare$] Pick $x_{n+1} \in \argmind{z\in X}{ f(z) + \phi_{n+1} (z)+\left( \frac{1}{2\sigma} +\bar{a}_n\right) \Vert z-x_n\Vert^2}$
    \end{itemize}
\end{itemize}}
\end{algorithm}
\vspace{0.5cm}
The dual update comes from Algorithm \ref{alg: lsc prox conjugate v1}, where we build an algorithmic update using $\Phi_{lsc}^\mathbb{R}$-subdifferentials.
Let us come to the convergence results, which are similar to Theorem \ref{thm: pd full lsc prox convergence}.

\begin{theorem}
\label{thm: pd full lsc prox convergence v1}
Let $f,g:X\to (-\infty,+\infty]$ be a proper $\lsc^\R$-convex. Let $(x_n)_{n\in\N},(\phi_n)_{n\in\N}, (\bar{x}_{n})_{n\in\N}$ and $(\bar{a}_n)_{n\in\N}$ be the sequences generated by Algorithm \ref{alg: lsc Chambolle-Pock full L=Id v2}. Assume that $\tau\sigma <1$ and $1+2\sigma \bar{a}_n > \sqrt{\sigma \tau}$ for all $n\in \N$, then we have
\begin{enumerate}[label=\roman*.]
    \item For any $(x,\phi)\in X\times \Phi_{lsc}^\mathbb{R}$, we have
    \begin{align}
 & \sum_{i=0}^{N}\mathcal{L}\left(x,\phi_{i+1}\right)-\mathcal{L}\left(x_{i+1},\phi\right) \nonumber\\
\geq & \left(\frac{1}{2\sigma}+\bar{a}_{N+1}\right)\left\Vert x-x_{N+1}\right\Vert ^{2}-\left(\frac{1}{2\sigma}+\bar{a}_{0}\right)\left\Vert x-x_{0}\right\Vert ^{2} \nonumber\\
& +\frac{1-\tau\sqrt{\sigma\tau}}{2\tau^{2}}\left(a-a_{N+1}\right)^{2}-\frac{1}{2\tau^{2}}\left(a-a_{0}\right)^{2}+\frac{1-\sqrt{\sigma\tau}}{2\tau} \left\Vert u-u_{N+1}\right\Vert ^{2}- \frac{1}{2\tau}\left\Vert u-u_{0}\right\Vert ^{2} \nonumber\\
 &+ \frac{\sqrt{\sigma\tau}}{2}\left( \frac{a-a_{N+1}}{\sqrt{\tau}} + \frac{\left\Vert x_{N+1}\right\Vert ^{2}-\left\Vert x_{N}\right\Vert ^{2}}{\sqrt{\sigma}} \right)^{2}  \nonumber\\
 & +\sum_{i=0}^{N}\left(\frac{1-\sqrt{\sigma\tau}}{2\sigma}+\bar{a}_{i}\right)\left\Vert x_{i}-x_{i+1}\right\Vert ^{2} 
 -\frac{\sqrt{\sigma\tau}}{2\sigma}\left(\left\Vert x_{i}\right\Vert ^{2}-\left\Vert x_{i+1}\right\Vert ^{2}\right)^{2} \nonumber\\
 & +\frac{\sqrt{\sigma\tau}}{2}\left( \frac{a_i-a_{i+1}}{\sqrt{\tau}} + \frac{\left\Vert x_{i-1}\right\Vert ^{2}-\left\Vert x_{i}\right\Vert ^{2}}{\sqrt{\sigma}}\right)^{2} +\frac{1-\sqrt{\sigma\tau}}{2\tau} \Vert u_i-u_{i+1}\Vert^2 \nonumber \\
    & +\frac{1}{2\tau^2}\left\Vert u_{n}-u_{n+1}\right\Vert ^{4} - \frac{\sqrt{\sigma\tau}}{2\tau} (a_i-a_{i+1})^2.
\label{eq: full lsc lagrange summable final form v1}
\end{align}
    where $\mathcal{L} (x,\phi) = f(x)+ \phi(x) -g^*_\Phi (\phi)$ is the Lagrangian.
    \item If the sequences $(x_n)_{n\in\mathbb{N}}$ and $\bar{a}_{n\in\mathbb{N}}$ are generated such that
    \begin{align}
    \label{eq: alg2 ass boundedness v1}
    \frac{1+2\sigma\bar{a}_{n}}{\sqrt{\sigma\tau}} &\geq 1+(\Vert x_n\Vert+\Vert x_{n+1}\Vert)^2\\ ,
    \frac{1}{\tau}\left\Vert u_{n}-u_{n+1}\right\Vert ^{4} &\geq  \sqrt{\sigma\tau} (a_i-a_{i+1})^2,
\end{align}
for all $n\in\mathbb{N}$. Then the sequences $(x_n,\phi_n)_{n\in\N}$ are bounded.
    \item In addition, if $X$ is finite dimensional, $(\bar{a}_n)_{n\in\N}$ converges in the way that 
    \[\lim_{n\to\infty} (1-\sqrt{\sigma\tau}+2\sigma \bar{a}_i) >0, \]
    then $(x_n,\phi_n)_{n\in\N}$ converge to a saddle point in the sense of \eqref{eq: kkt condition lsc convex}.
\end{enumerate}
\end{theorem}
The proof for Theorem \ref{thm: pd full lsc prox convergence v1} can be derived from the proof of Theorem \ref{thm: pd full lsc prox convergence}.


\section{Numerical Examples}
\label{sec: lsc-pd numerical}

\subsection{Toy Examples}
\label{subsec: example 1-2}

In this subsection, we demonstrate the performance of the primal-dual algorithms (Algorithm \ref{alg: lsc Chambolle-Pock}, \ref{alg: lsc Chambolle-Pock full L=Id} and \ref{alg: lsc Chambolle-Pock full L=Id v2}) for the following toy example.
\begin{example}
\label{ex1}
Consider the following problem on the real line,
\begin{equation}
    \min_{x\in \mathbb{R}} x^4 +x^2.
    \label{prob 1}
\end{equation}
This is a convex problem with a unique solution at $x=0$. We define the class $\Phi_{lsc}^\mathbb{R}$ of quadratic function,
\[
\Phi_{lsc}^\mathbb{R}=\{ \phi=(a,u)\in \mathbb{R}^2: \phi(x) = -ax^2 +ux\}.
\]
We split problem \eqref{prob 1} into the form $f+g$, where $f(x)=x^4, g(x) = x^2$. Since $g$ is convex, the corresponding Lagrange saddle point problem \eqref{prob: Lagrange g convex} is
\begin{equation}
\label{ex1: Lagrange prob convex}
    \min_{x\in\mathbb{R}}\max_{y\in \mathbb{R}} f(x)+\langle x,y\rangle -g^*(y) = \min_{x\in\mathbb{R}}\max_{y\in \mathbb{R}} x^4 +xy - \frac{y^2}{4}.
\end{equation}
Problem \eqref{ex1: Lagrange prob convex} is equivalent to \eqref{prob 1} in the sense of the same optimal solution and optimal value.
By using $\Phi_{lsc}^\mathbb{R}$-conjugation, we obtain $\Phi_{lsc}^\mathbb{R}$-saddle point problem \eqref{prob: DCP L=Id}
\begin{equation}
\label{ex1: Lagrange prob lsc convex}
\min_{x\in\mathbb{R}}\max_{\phi\in \Phi_{lsc}^\mathbb{R}} f(x)+\phi(x) -g^*_\Phi (\phi) = \min_{x\in\mathbb{R}}\max_{(a,u)\in \mathbb{R}^2} x^4 -ax^2+ux-g^*_\Phi (a,u),
\end{equation}
where $g^*_\Phi$ takes the form
\[
g^*_\Phi (a,u) =\begin{cases}
0 & a=-1,u=0,\\
\frac{u^{2}}{4\left(a+1\right)} & a>-1,\\
+\infty & \text{otherwise}.
\end{cases}
\]
Problem \eqref{ex1: Lagrange prob convex} has $(x,y) = (0,0)$ as a saddle point while problem \eqref{ex1: Lagrange prob lsc convex} has $(x,\phi)=(0,(-1,0))$ as the saddle point.

We implement Chambolle-Pock algorithm \cite{chambolle2011first} and Algorithm \ref{alg: lsc Chambolle-Pock} to Problem \ref{ex1: Lagrange prob convex} while using Algorithms \ref{alg: lsc Chambolle-Pock full L=Id} and \ref{alg: lsc Chambolle-Pock full L=Id v2} to problem \eqref{ex1: Lagrange prob lsc convex}. Since the dual variables belong to different spaces, we only visualize the primal variable $x_n$ returned by these algorithms in Figure \ref{ex1:fig}, when we apply the same set of parameters.

We set $\tau =\sigma=0.25$, initials $x_{-1} = x_0 = 5,y_0=10,\phi_0 =(2,-5)$, maximum iteration number is $N=501$. For algorithms \ref{alg: lsc Chambolle-Pock}, \ref{alg: lsc Chambolle-Pock full L=Id} and \ref{alg: lsc Chambolle-Pock full L=Id v2}, we choose $\bar{a}_0 = 100$. The results are shown in Figure \ref{ex1:fig}.
\begin{figure}[h]
\centering     
\includegraphics[width=0.9\textwidth]{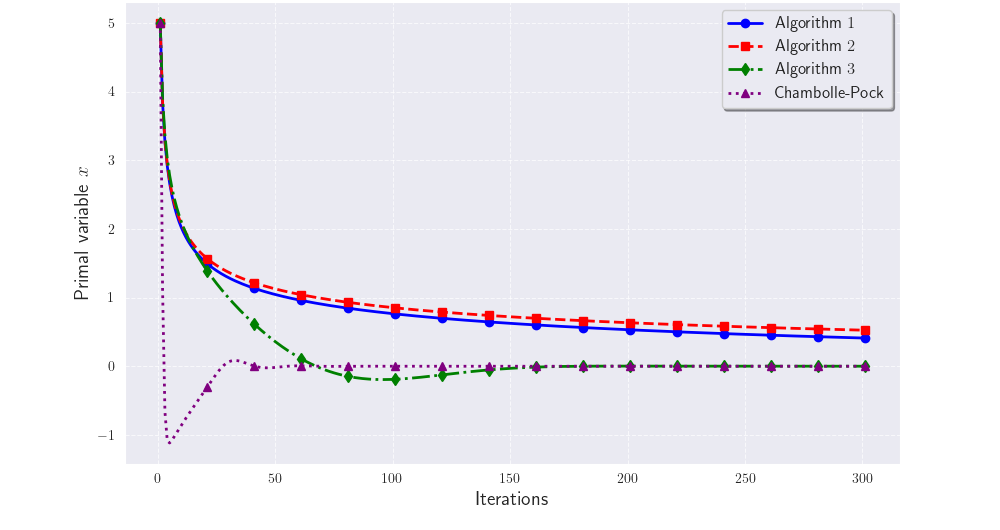}
    \caption{Performance of Primal-dual algorithms for Example~\ref{ex1}.}
    \label{ex1:fig}
\end{figure}

The details and explicit forms of the updates will be given in the appendix in Section \ref{apendix: numerical 1}.
\end{example}

Since the framework of $\Phi_{lsc}^\mathbb{R}$-conjugate applies to nonconvex function, we test algorithms \ref{alg: lsc Chambolle-Pock full L=Id} and \ref{alg: lsc Chambolle-Pock full L=Id v2} on the following example.
\begin{example}
\label{ex2: nonconvex example}
We want to find a solution to the problem,
\begin{equation}
\label{ex2: prob}
    \min_{x\in \mathbb{R}} x^4 -x^2.
\end{equation}
This problem is nonconvex (specifically weakly convex) with two global minimizers at $x=\pm\frac{1}{\sqrt{2}}$. We consider problem \eqref{ex2: prob} as $\min_x f(x)+g(x)$ with $f(x)=x^4,g(x)=-x^2$, and the $\Phi_{lsc}^\mathbb{R}$-saddle point problem has the form 
\[
\min_{x\in \mathbb{R}} \max_{\phi\in \Phi_{lsc}^\mathbb{R}} f(x)+\phi(x)-g^*_\Phi (\phi),
\]
which has two saddle points at $(\pm 1/\sqrt{2}, (1,0))$.
The function $g^*_\Phi (\phi)$ can be calculated as analoguously, which is
\[
g^*_\Phi (a,u) =\begin{cases}
0 & a=1,u=0,\\
\frac{u^{2}}{4\left(a-1\right)} & a>1,\\
+\infty & \text{otherwise}.
\end{cases}
\]
In fact, we provide the calculation of $g^*_\Phi$ for a general function $g(x)=cx^2, c\in \mathbb{R},$ in the appendix.

In this example, we use algorithms \ref{alg: lsc Chambolle-Pock full L=Id} and \ref{alg: lsc Chambolle-Pock full L=Id v2} to solve the $\Phi_{lsc}^\mathbb{R}$-saddle point problem. 
Let us set $\tau =\sigma=0.25$, maximum iteration number is $N=501$ and $\bar{a}_0 = 200$. We test the following initializations: Case $1$: $x_{-1} = x_0 = 2,\phi_0 =(1.5,2)$; Case 2: $x_{-1} = x_0 = 5,\phi_0 =(1.5,2)$; Case 3: $x_{-1} = x_0 = 5,\phi_0 =(5,0)$.
The results are shown in Figure \ref{ex2:fig}.

\begin{figure}[h]
\centering     
\subfigure[Case 1]{\label{}\includegraphics[height=4.cm]{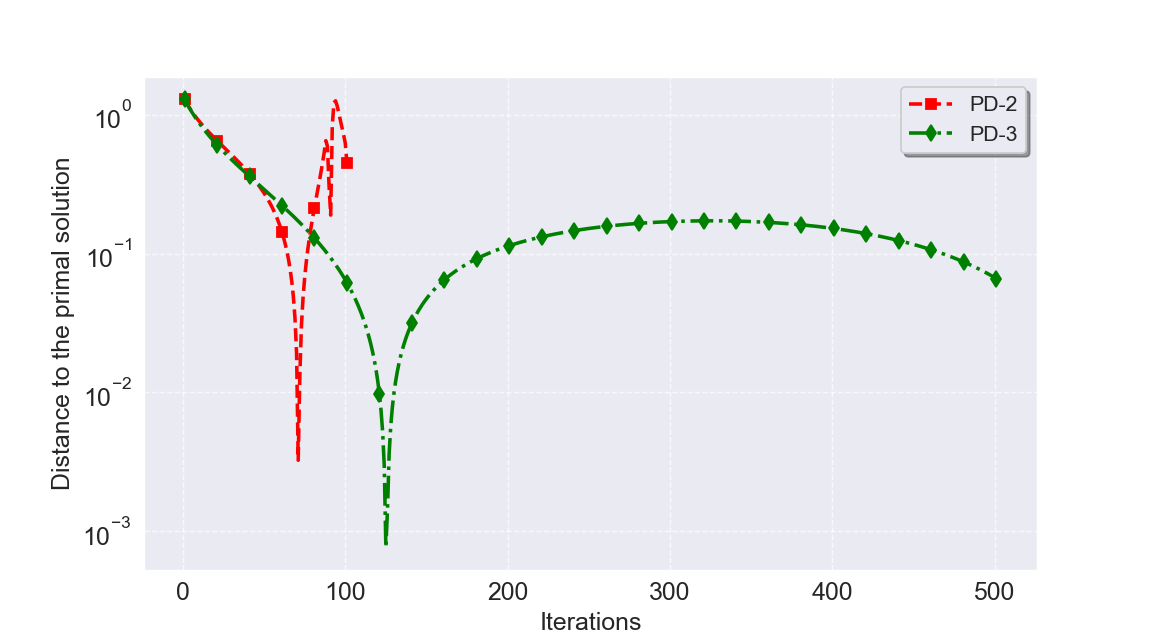}}
\subfigure[Case 2]{\label{}\includegraphics[height=4.cm]{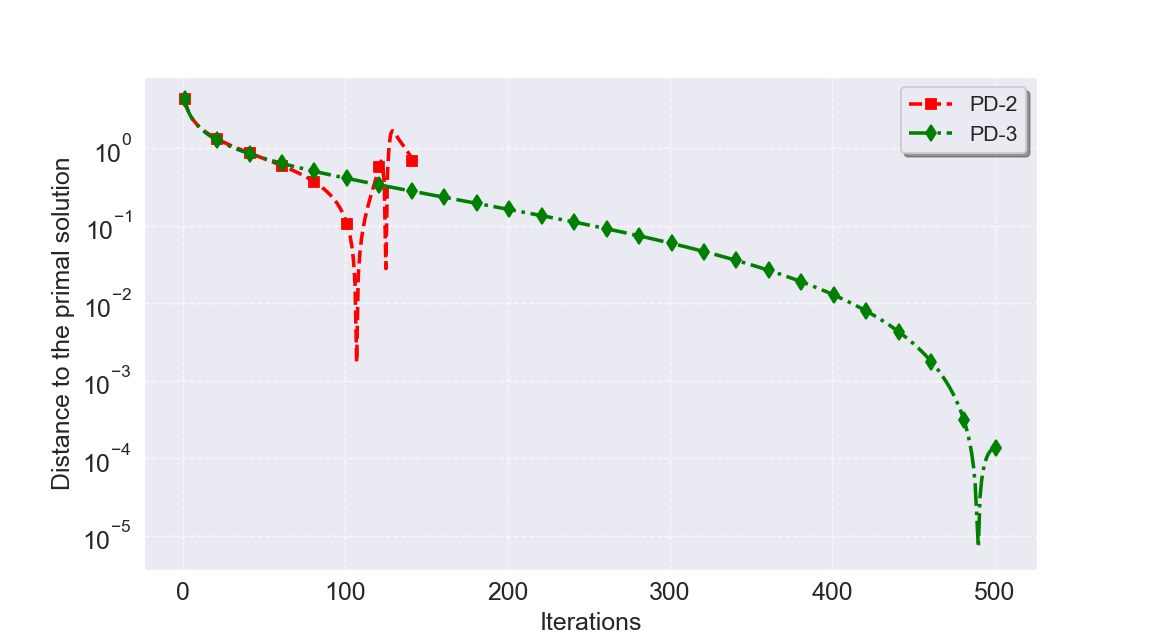}}
\subfigure[Case 3]{\label{}\includegraphics[height=4.cm]{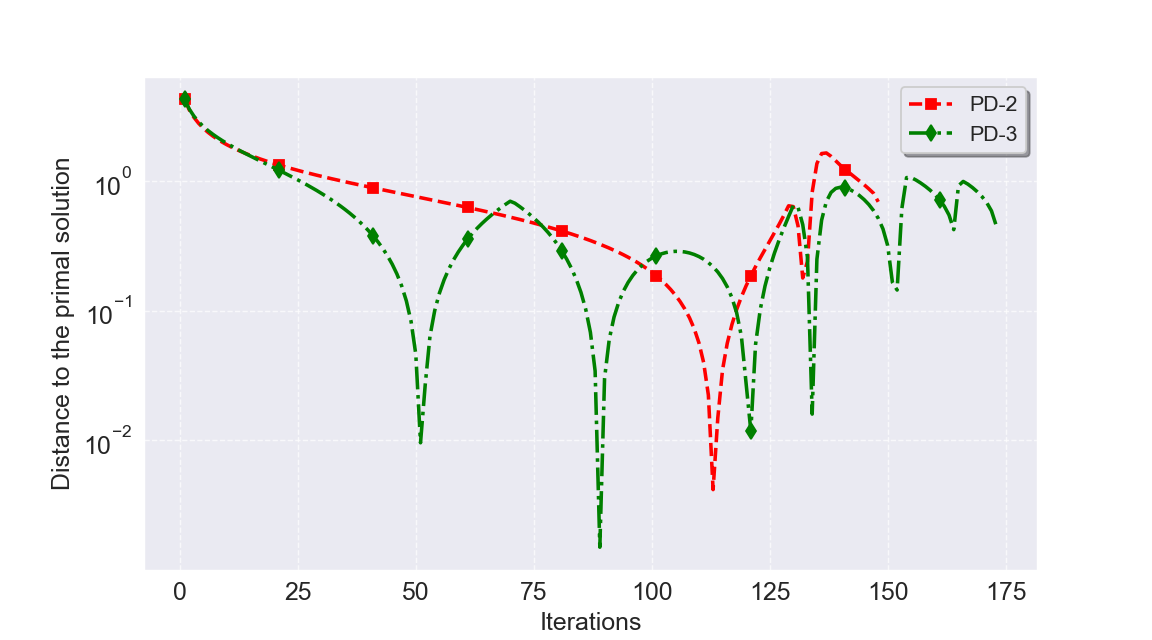}}
    \caption{$\lsc^\R$-Primal Dual algorithms for Example~\ref{ex2: nonconvex example}.}
    \label{ex2:fig}
\end{figure}
In all cases, Algorithm \ref{alg: lsc Chambolle-Pock full L=Id v2} outperforms Algorithm \ref{alg: lsc Chambolle-Pock full L=Id}. Both algorithms oscillate with a sharp jump, but Algorithm \ref{alg: lsc Chambolle-Pock full L=Id} stops early compare to Algorithm \ref{alg: lsc Chambolle-Pock full L=Id v2}. The oscillating behaviors are better shown in Case $3$. For Case $1$, we want to see how the iterates behave when starting near the solution, but the iterates converge slower than Case $2$ where $x_0$ is farther to the solutions. On the other hand, the pattern in Case $3$ indicates that the iterates do not tend to the solution but are bounded by a certain threshold away from the solutions. We suspect that for nonconvex problem, the initializations are important for the convergence of the iterate. On the other hand, each proximal step involves solving a cubic equation in which we cannot control the cubic solutions (more details in Appendix \ref{apendix: numerical 1}). On the other hand, the stopping criterion causes us to miss the best iterate in term of the distance to the solution.
\end{example}

\subsection{Application}
\label{subsec: example 3}
\begin{example}[\textbf{Binary tomography}]
Binary Tomography (BT) is a special case of Discrete Tomography (DT) which aims at reconstructing binary images starting from a limited number of their projections \cite{schule2005discrete, kadu2019convex}. 
The image $\overline x \in \R^n$ is represented as a grid of $n=n_1\times n_2$ pixels taking values $x_j \in \{-1,1\}$ for $j=1,\dots,n$. The projections $y_i$ for $i=1,\dots,m$ are linear combinations of the pixels along $m$ directions. The linear transformation from image to projections is modelled as 
\begin{equation}
    y=\mathbf{A}\overline x + \omega,
\end{equation}
where $x_j$ denotes the value of the image in the $j$-th cell, $y_i$ is the weighted sum of the image along the $i$-th ray, the element $\mathbf{A}_{i,j}$ of matrix $\mathbf{A}\in\R^{m\times n}$ is proportional to the length of the $i$-th ray in the $j$-th cell and finally $\omega\in\R^n$ is an additive noise with Gaussian distribution and standard deviation $\sigma>0$. In general, the projection matrix $\mathbf A$ has a low rank, meaning that $\operatorname{rank}(\mathbf A) < \min\{ n,m\}$ \cite{kadu2019convex}. We model the problem as
\begin{equation}\label{eq:discrete_tomography_objective}
    \minimize{{x\in\R^n, x_i\in \{ -1,1\} }}  \frac{\|\mathbf Ax-y\|_2^2}{2}.
\end{equation}
Instead of considering the binary variable $x$, we can turn problem \eqref{eq:discrete_tomography_objective} into a continuous one by adding the function $\fonc{F}{\mathbb{R}^n}{\mathbb{R}}$ defined as
\begin{equation}\label{eq:binaryfunction}
F({x}) = {\sum_{i=1}^n |x_i^2-1|},\end{equation}
which is separable as the sum of functions $f_i (x) = |x_i^2 -1|$. Observe that $F$ is a weakly convex function with the set of minimisers 
\[{ S_F = \{x=(x_1,\dots,x_n)\in\R^n\,|\, x_i \in\{-1,1\} \;\text{for}\; i=1,\dots,n\} = \{-1,1\}^n},\]
and can be seen as a non-smooth counterpart to the function 
 ${ x\mapsto  {\sum_{i=1}^n x_i^2(x_i-1)^2} }$,  $x\in\R^n$,
proposed in \cite{yang2021doubly} to promote binary integer solutions (with values in $\{0,1\}^n$) in Markov Random Field approaches. To the best of our knowledge, this is the first work where function $F$ is used in such a context.

Problem \eqref{eq:discrete_tomography_objective} transforms into
\begin{equation}\label{eq:discrete_tomography_new}
    \minimize{{x\in\R^n }}  \frac{\|\mathbf Ax-y\|_2^2}{2} +F({x}).
\end{equation}
We can write the objective function of \eqref{eq:discrete_tomography_new} in the form of composite functions $F(x)+G(Ax)$ and implement $\lsc^\R$-Chambolle-Pock Algorithm \ref{alg: lsc Chambolle-Pock} to solve this problem. 

The function $G:\R^m \to \R$ is $G(z) = \frac{\Vert z-y\Vert^2}{2}$ which has the conjugate
\[
G^*(v) = \frac{\Vert v\Vert^2}{2}+\langle v,y\rangle,
\]
and we can write the dual update similar to the previous examples, which is
\[
v_{n+1} = \frac{v_n +\tau (A\overline{x}_n - y)}{\tau+1}.
\]
For the primal step, it can be calculated as a proximal step as we mention in Algorithm \ref{alg: lsc Chambolle-Pock}. Specifically, since the function $F$ is weakly convex, we can obtain an explicit form of convex proximal operator for each component of $F$ i.e. the function $\mathcal{F}(x) = |x^2-1|$ (see \cite[Remark 8]{bednarczuk2023convergence}). For $\alpha >0, y\in \mathbb{R}$, we have
\begin{equation}\label{eq:prox_of_f}
\quad (\forall y \in \R  )\quad 
    \prox_{\alpha \mathcal{F}} (y) = \begin{cases}
        \frac{y}{1+2\alpha} & \text{ if } | y | >1+2\alpha,\\
        \frac{y}{1-2\alpha} & \text{ if } | y | <1-2\alpha,\\
        \frac{y}{|y|} & \text{ otherwise. } 
    \end{cases}  
\end{equation}

\begin{remark}
    The model in \eqref{eq:discrete_tomography_objective} (for appropriate choices of matrix $\mathbf A\in\R^{m\times n}$, vector $y\in\R^m$ and possible additional convex constraints) could be applied to other Binary Quadratic Programs (a special class of QCQP problems having the equality constraint $x_i^2=1$ for every $i\in\{1, \dots, n\}$ ) arising in Computer Vision, such as Graph Bisection, Graph Matching and Image Co-Segmentation (see \cite[Table 2]{Wang2016_BQ} and the references therein).
\end{remark}

\paragraph{Setting:} For $\lsc^\R$-Chambolle-Pock Algorithm, we fix the number of iteration to $500$, stepsizes $\tau=\sigma= \frac{0.9}{ \Vert A\Vert}$ to ensure that $\tau\sigma \Vert A\Vert^2 <1$. For the parameter $\bar{a}_n$, we start with $a_0 =100$ and updating $\bar{a}_{n+1} = \bar{a}_n -\frac{1}{n^2}$.
\paragraph{Numerical tests} For our simulations we used the MATLAB codes and data from \cite{kadu2019convex, githubpage}. We reproduced the same synthetic setting and test from simple to complex images: Consider multiples phantoms (\emph{Apple}, \emph{Bird}, \emph{Lizard}, \emph{Octopus}, \emph{Spring}, \emph{Horse}, \emph{Pocket Watch} and \emph{Blood Vessels} ), corresponding to binary images of size 64 x 64 pixels. Operator $\mathbf A$ models an X-ray tomographic scan with 64 detectors and a parallel beam acquisition geometry with four angles (0°, 50°, 100°, 150°). We set $\sigma=0.01$ for the additive noise. \autoref{fig:discrete_tomography} illustrates, from left to right, the original image and the reconstructions obtained with the Least Squares QR method (LSQR), with the Truncated Least Squares QR method (TLSQR), with the DUAL method proposed in \cite{kadu2019convex} and finally our proposed Algorithm \ref{alg: lsc Chambolle-Pock} applied to \eqref{eq:discrete_tomography_new}, which we refer to as $\lsc^\R$-Chambolle-Pock method ($\lsc^\R$-CP). All methods are initialised with a vector of zeros.

\begin{figure}[ht]
    \centering
    \begin{tabular}{p{0.16\textwidth}p{0.16\textwidth}p{0.16\textwidth}p{0.16\textwidth}p{0.16\textwidth}}
    \phantom{ccccc} ORIG & \phantom{cccc} LSQR & \phantom{cc} TLSQR & \phantom{ci} DUAL & \phantom{c}{$\lsc^\R$-CP} \\
  \multicolumn{5}{c}{\centering
\includegraphics[width=0.9\textwidth,trim={3cm 9cm 2cm 9cm},clip]{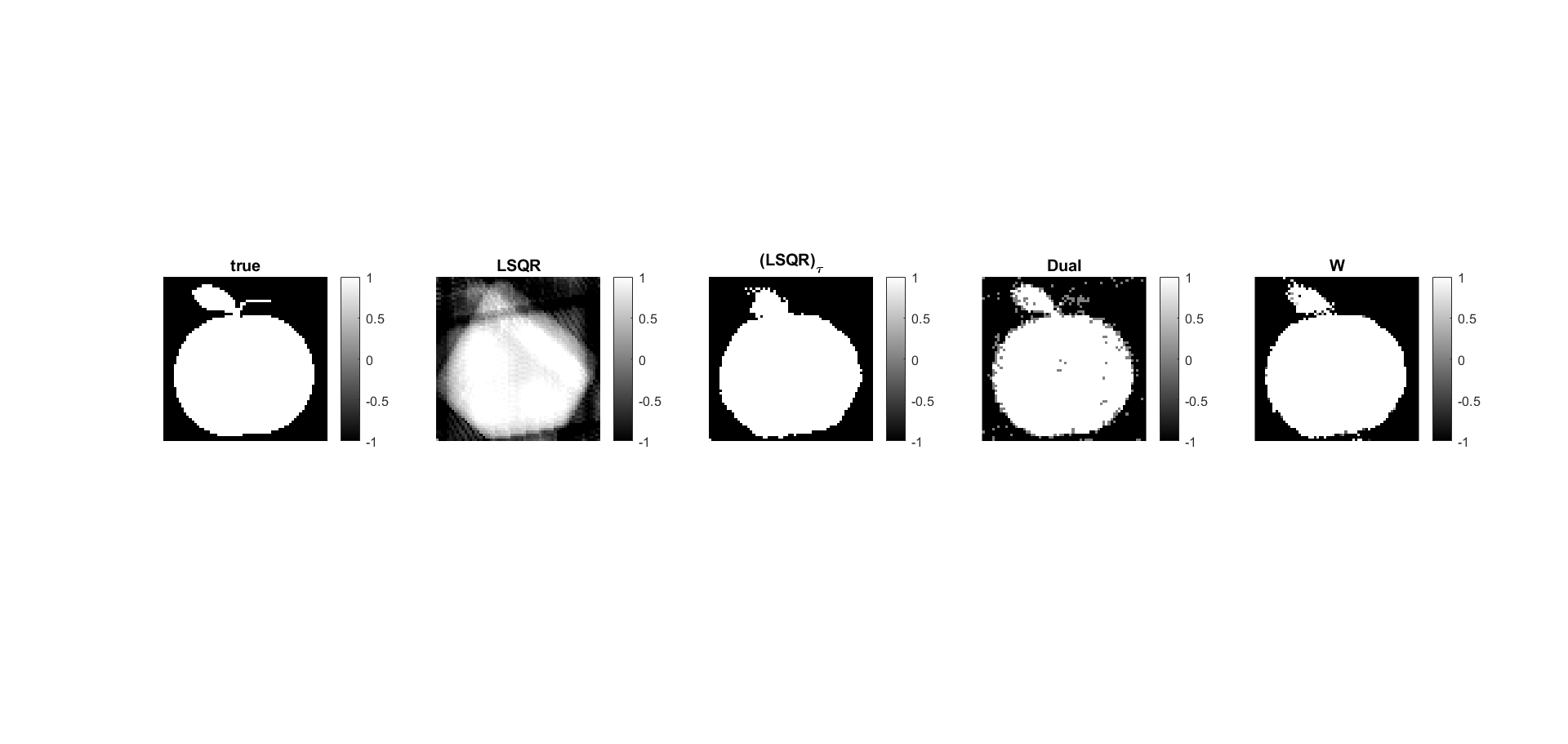}}\\
   \multicolumn{5}{c}{ 
\includegraphics[width=0.9\textwidth,trim={3cm 9cm 2cm 9.5cm},clip]{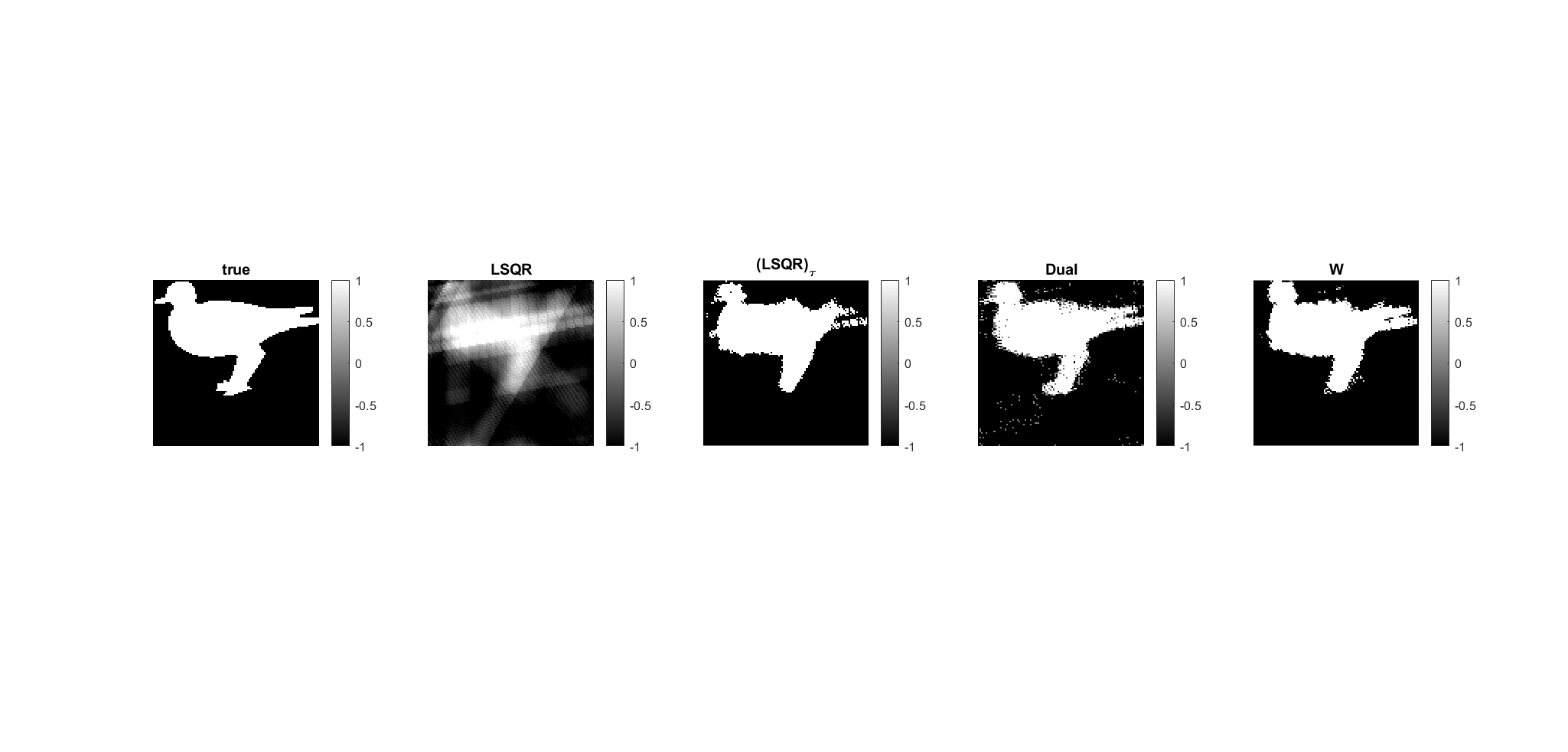}}\\
   \multicolumn{5}{c}{ 
\includegraphics[width=0.9\textwidth,trim={3cm 9cm 2cm 9.5cm},clip]{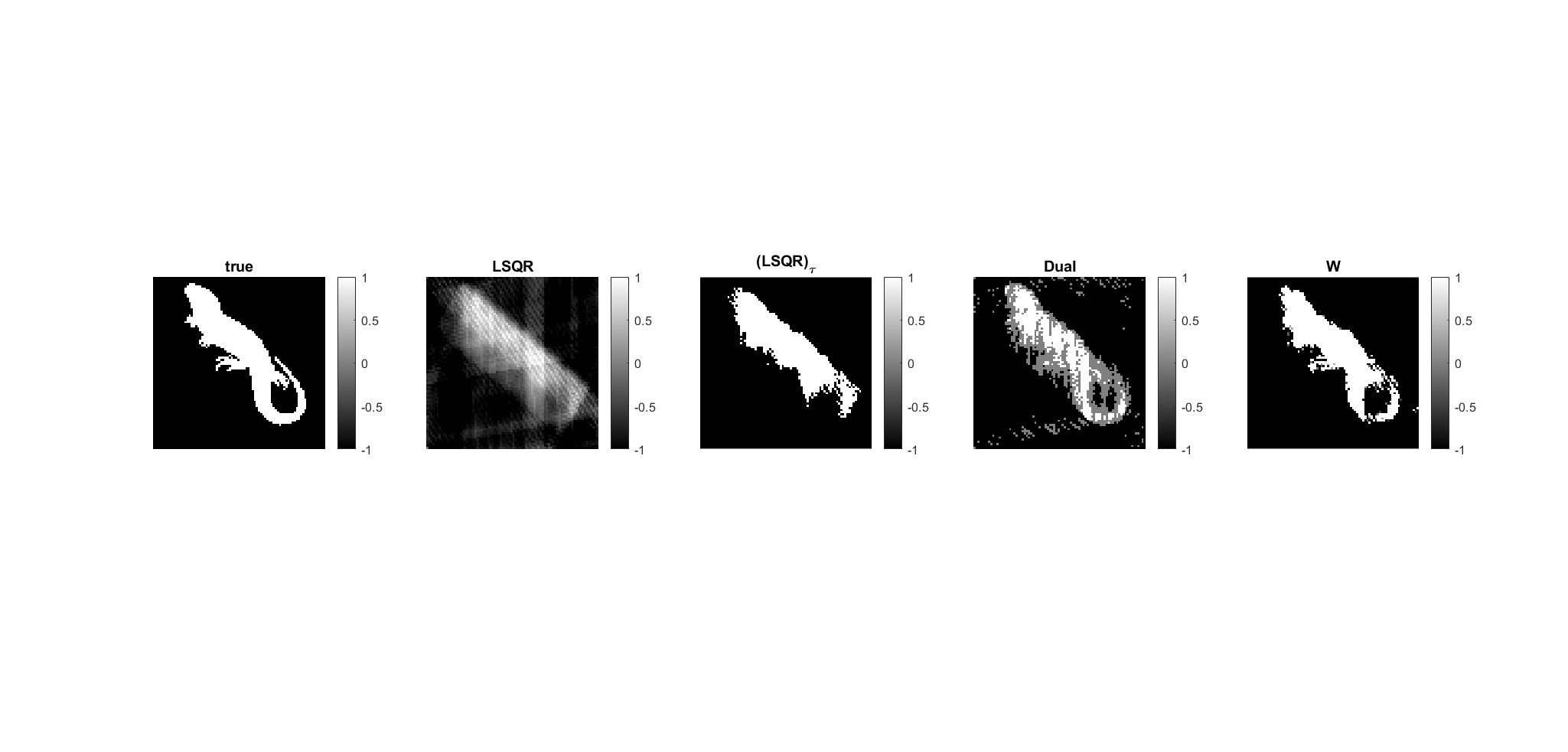}}\\
   \multicolumn{5}{c}
{\includegraphics[width=0.9\textwidth,trim={3cm 9cm 2cm 9cm},clip]{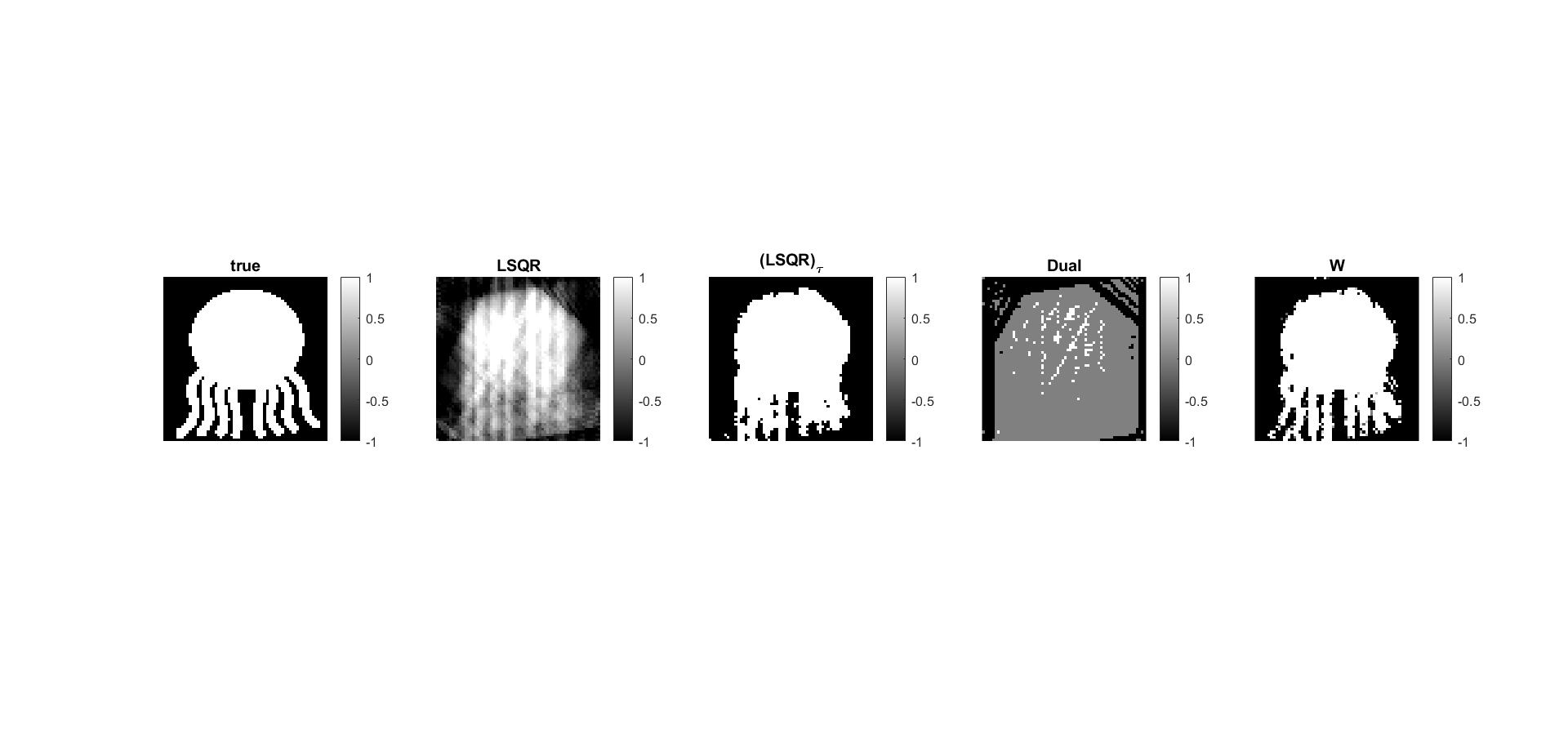}}\\
   \multicolumn{5}{c}{
\includegraphics[width=0.9\textwidth,trim={3cm 9cm 2cm 9.5cm},clip]{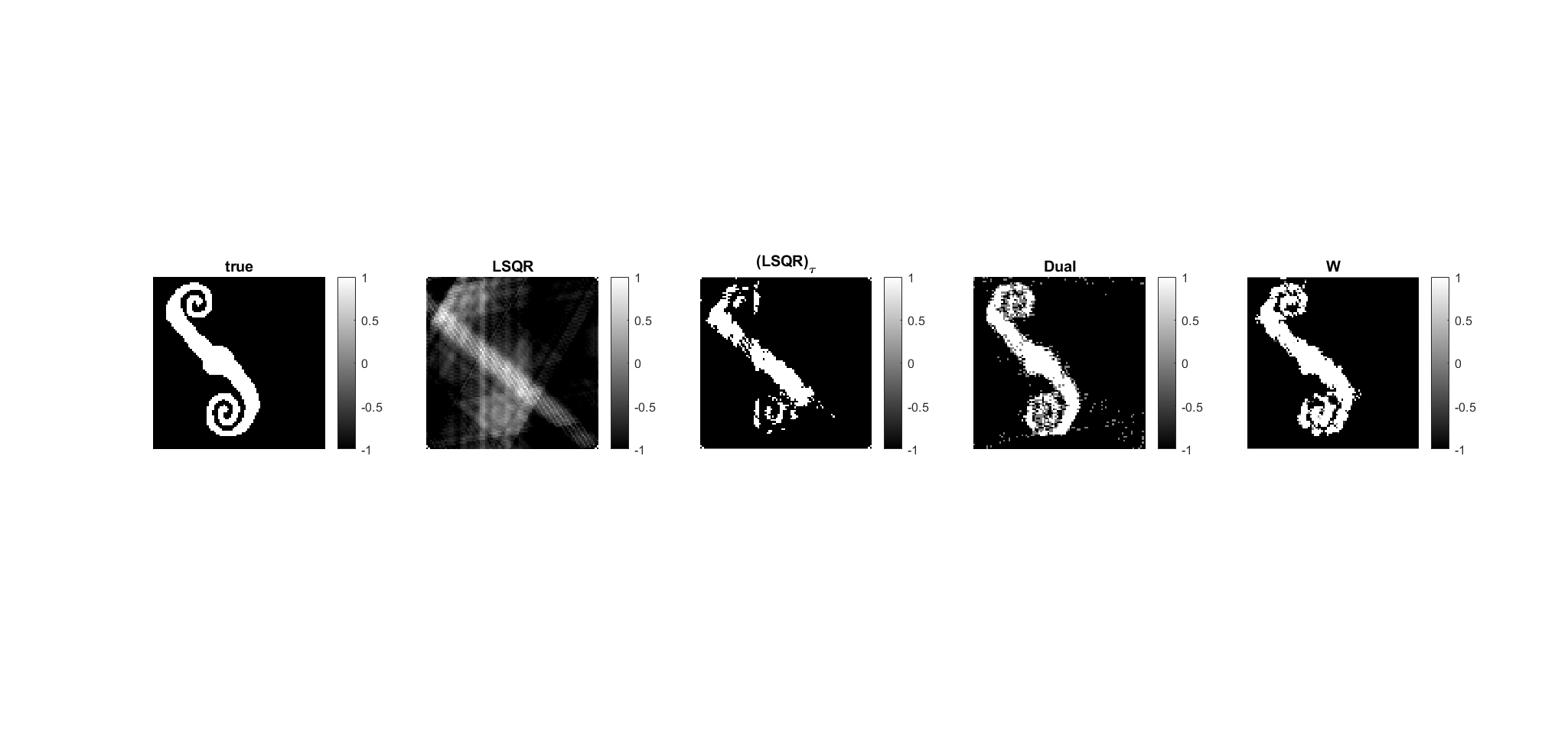}}\\
    \multicolumn{5}{c}{
\includegraphics[width=0.9\textwidth,trim={3cm 9.5cm 2cm 8.9cm},clip]{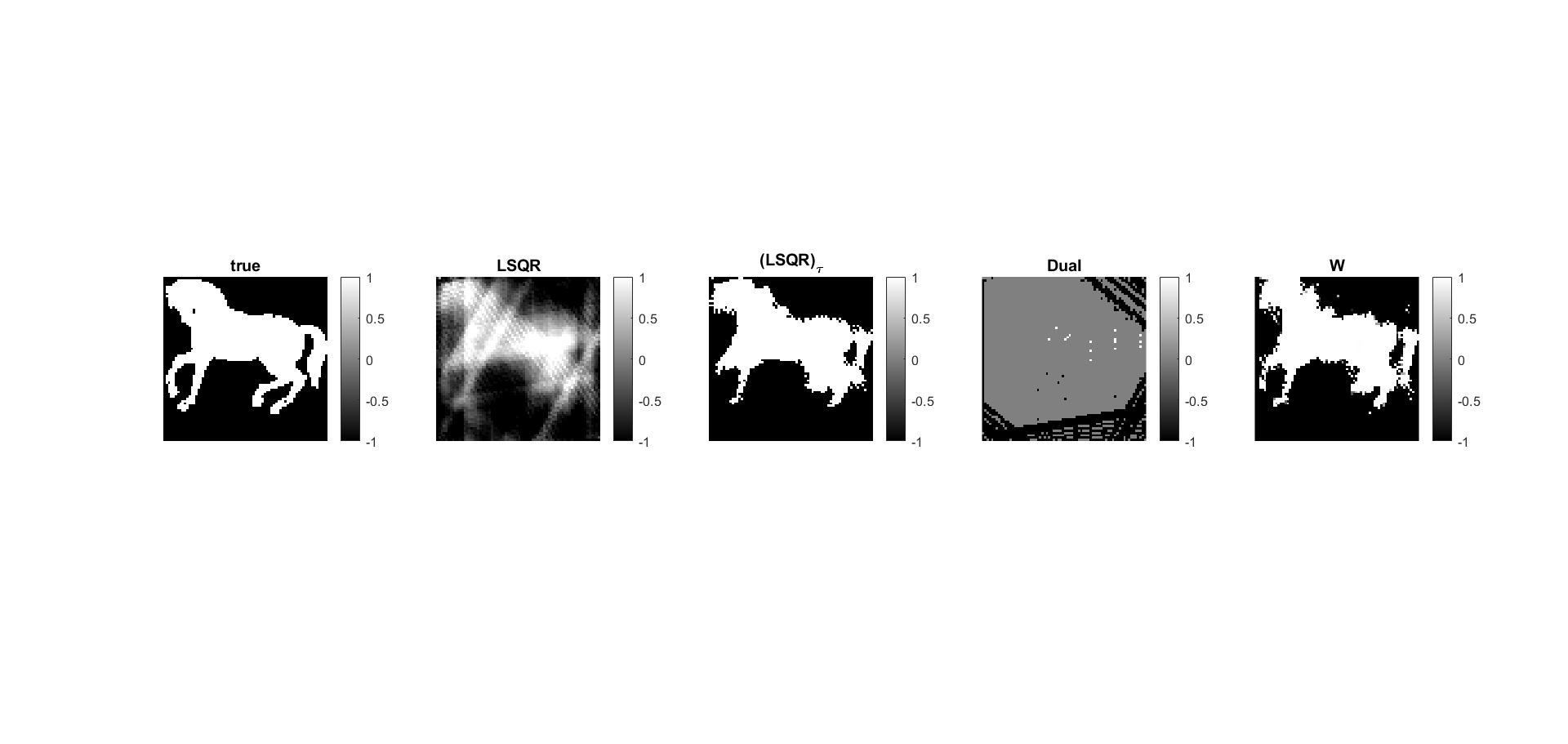}}\\
\multicolumn{5}{c}{
\includegraphics[width=0.9\textwidth,trim={3cm 9.5cm 2cm 9.5cm},clip]{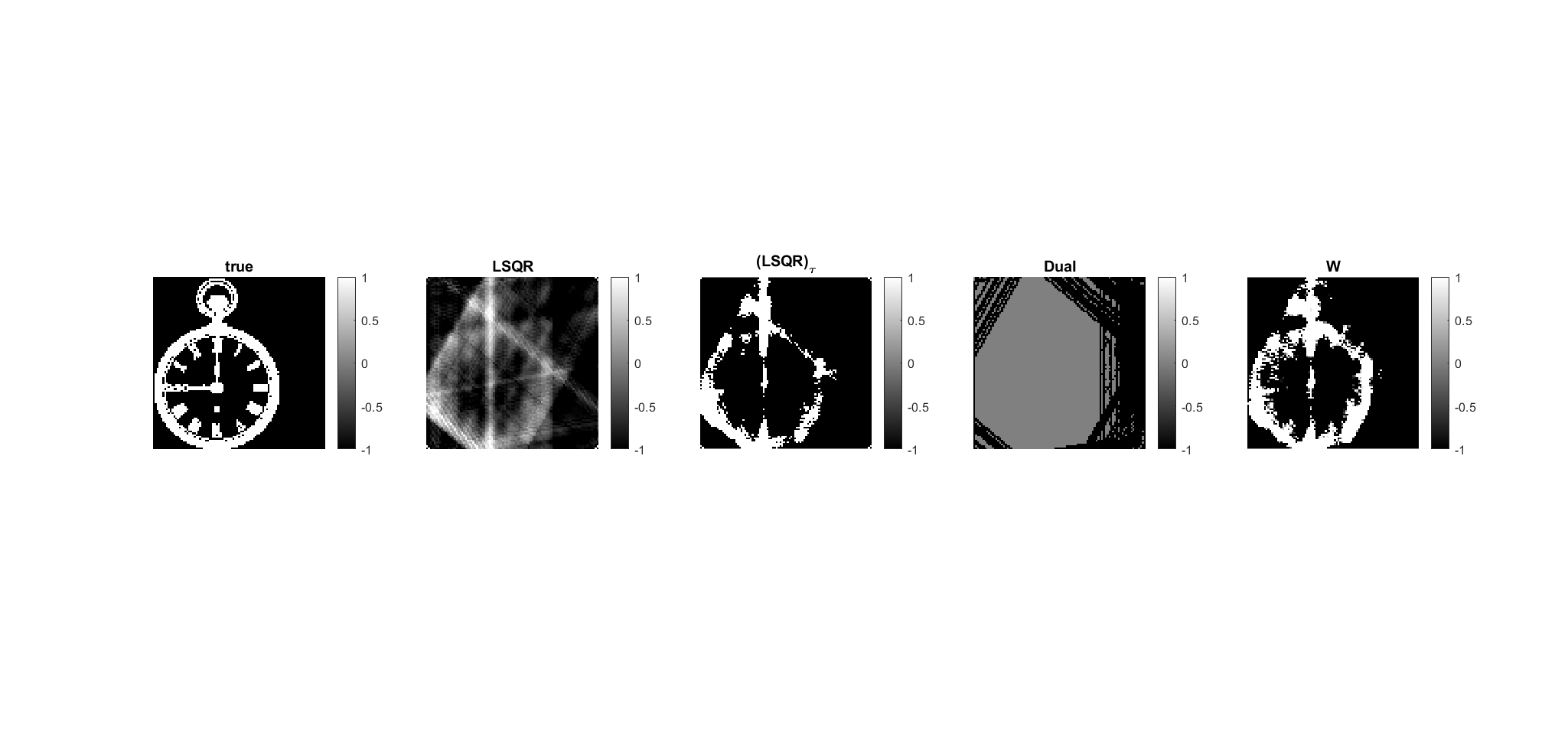}}\\
\multicolumn{5}{c}{
\includegraphics[width=0.9\textwidth,trim={3.2cm 9cm 2cm 9cm},clip]{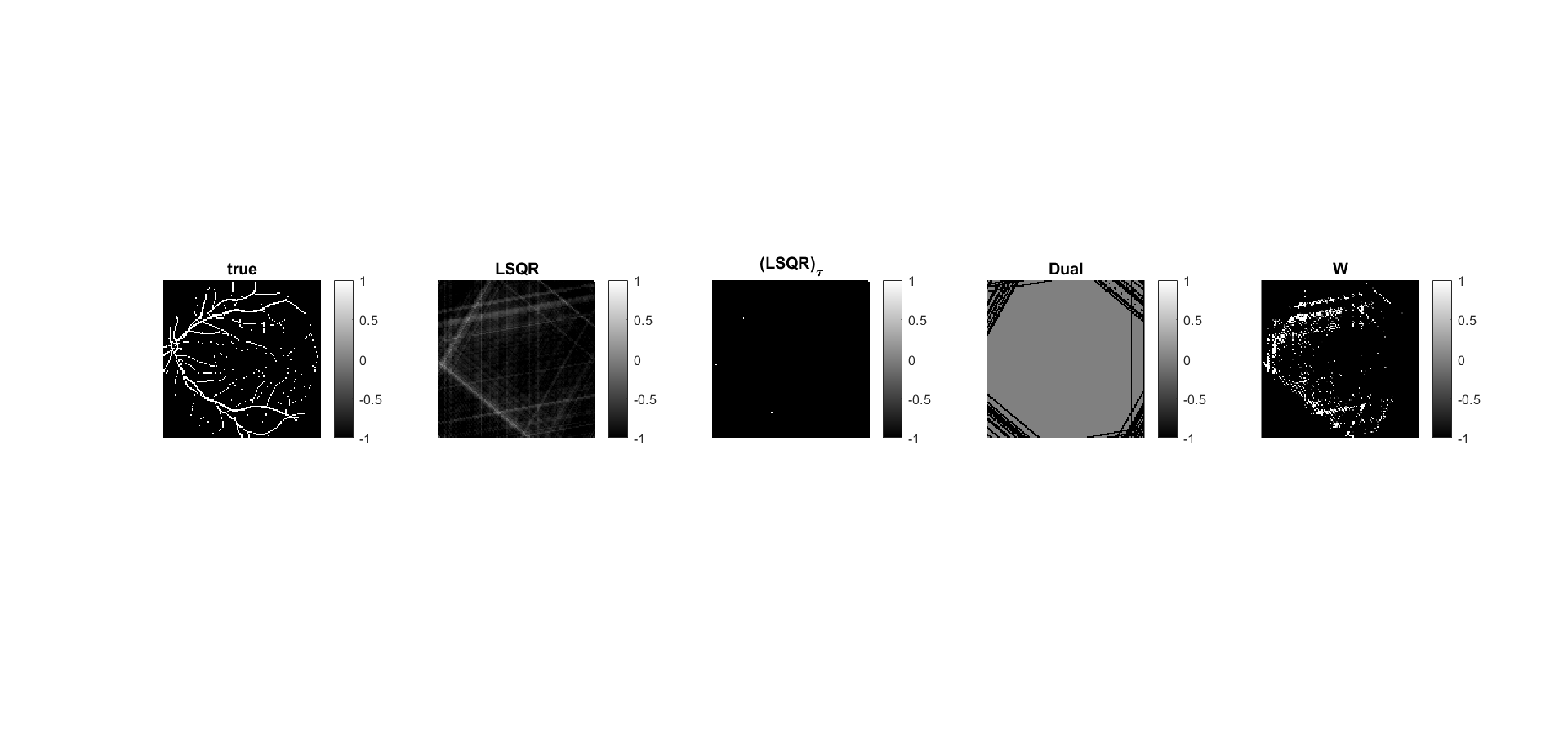}}\\
      \end{tabular}
       \caption{Comparisons between different solutions. From left to right: original image, least squares solution, thresholded least squares solution, solution obtained with the dual method proposed in \cite{kadu2019convex}, solution obtained with our model.}
    \label{fig:discrete_tomography}
\end{figure}

\end{example}
\paragraph{Funding}{This work was funded by the European Union's Horizon 2020 research and innovation program under the Marie Sk{\l}odowska-Curie grant agreement No. 861137. This work represents only the authors' view and the European Commission is not responsible for any use that may be made of the information it contains.}

\bibliographystyle{acm}
\bibliography{ref_library}

\begin{appendices}
\section{Solving minimization problem on the class of elementary functions}
\label{appendix: conjugate prox}
In this part, we consider how to solve the following problem 
\begin{equation}
    \min_{\phi \in \Phi_{lsc}^\mathbb{R}} h (\phi),
    \label{eq: min conjugate g}
\end{equation}
where $h:\Phi_{lsc}^\mathbb{R} \to (-\infty,+\infty]$ is a proper lsc function which is bounded from below.
We want to design a proximal algorithm to solve problem \eqref{eq: min conjugate g} using $\Phi_{lsc}^\mathbb{R}$-subdifferentials \eqref{eq: lsc subgrad conj}. This leads us to Algorithm \ref{alg: lsc prox conjugate}.

\begin{algorithm}[H]\caption{$\lsc^\R$-Proximal Algorithm on $\Phi_{lsc}^\mathbb{R}$}\label{alg: lsc prox conjugate}
\textbf{Initialize:} Choose $\gamma >0, \phi_0=(a_0,u_0)\in \Phi_{lsc}^\mathbb{R}$\\
{\normalfont\textbf{Update: }}{
For $n\in\N$,
\begin{itemize}
    \item Find $\phi_{n+1}=(a_{n+1},u_{n+1})$ such that 
    \begin{equation}
    \label{eq: proximal update conjugate}
    \frac{u_n -u_{n+1}}{\gamma} \in \partial_{\Phi}^\mathbb{R} h(\phi_{n+1}).
    \end{equation}
    \item \textbf{If} there is no such $\phi_{n+1}$, \textbf{return} $\phi_n$.
\end{itemize}}
\end{algorithm}
\vspace{0.5cm}

Algorithm \ref{alg: lsc prox conjugate} follows the approach of \cite{bednarczuk2025proximal}. Since $\Phi_{lsc}^\mathbb{R}$-subgradient is defined on $X$, the simplest choice would be to use only the second component of $\phi$ which is $u\in X$.
The first component $a\in \mathbb{R}$ can be considered as a parameter which we can choose to fit the proximal update \eqref{eq: proximal update conjugate}.
Let us state the first estimation comes from Algorithm \ref{alg: lsc prox conjugate}
\begin{proposition}
\label{prop: lsc prox conj 1st estimation}
Let $h:\Phi_{lsc}^\mathbb{R} \to (-\infty,+\infty]$ be a proper lsc function on $\Phi_{lsc}^\mathbb{R}$ which is bounded from below. Let $(\phi_{n})_{n\in\mathbb{N}}$ be the sequence generated by Algorith \ref{alg: lsc prox conjugate} with stepsize $\gamma>0$. Then for any $\phi\in \Phi_{lsc}^\mathbb{R}$, we have, 
\begin{equation}
\label{eq: PPA conjugate 1st estimate}
h\left(\phi\right)-h\left(\phi_{n+1}\right) \geq 
\frac{1}{2\gamma}\left[\left\Vert u-u_{n+1}\right\Vert ^{2}-\left\Vert u-u_{n}\right\Vert ^{2}\right]+\left(1-\frac{2\left(a-a_{n+1}\right)}{\gamma}\right)\frac{\left\Vert u_{n}-u_{n+1}\right\Vert ^{2}}{2\gamma}.
\end{equation}
If $\gamma \geq a_n -a_{n+1}$ for all $n\in\mathbb{N}$, then $h(\phi_n)$ is non-increasing.
\end{proposition}

\begin{proof}
Expressing \eqref{eq: proximal update conjugate} using $\Phi_{lsc}^{\mathbb{R}}$-subgradient, for any $\phi\in\Phi_{lsc}^{\mathbb{R}}$, we obtain \eqref{eq: PPA conjugate 1st estimate},
\begin{align*}
h\left(\phi\right)-h\left(\phi_{n+1}\right) & \geq-\frac{\left(a-a_{n+1}\right)}{\gamma^{2}}\left\Vert u_{n}-u_{n+1}\right\Vert ^{2}+\frac{1}{\gamma}\left\langle u-u_{n+1},u_{n}-u_{n+1}\right\rangle \\
 & =\frac{1}{2\gamma}\left[\left\Vert u-u_{n+1}\right\Vert ^{2}-\left\Vert u-u_{n}\right\Vert ^{2}\right]+\left(1-\frac{2\left(a-a_{n+1}\right)}{\gamma}\right)\frac{\left\Vert u_{n}-u_{n+1}\right\Vert ^{2}}{2\gamma}.
\end{align*}

Let $\phi=\phi_{n}$ in \eqref{eq: PPA conjugate 1st estimate}, we obtain 
\begin{align*}
h\left(\phi_{n}\right)-h\left(\phi_{n+1}\right) & \geq\left(1-\frac{\left(a_{n}-a_{n+1}\right)}{\gamma}\right)\frac{\left\Vert u_{n}-u_{n+1}\right\Vert ^{2}}{\gamma},
\end{align*}
which the RHS is non-negative when $\gamma \geq a_n-a_{n+1}$. Hence we finish the proof.
\end{proof}
To obtain convergence of Algorithm \ref{alg: lsc prox conjugate}, $(a_n)_{n\in\mathbb{N}}$ has to be chosen in a specific way as follow.
\begin{theorem}
\label{thm: prox conj convergence}
Let $h:\Phi_{lsc}^\mathbb{R} \to (-\infty,+\infty]$ be a proper lsc function on $\Phi_{lsc}^\mathbb{R}$ which is bounded from below. Let $(\phi_{n})_{n\in\mathbb{N}}$ be the sequence generated by Algorithm \ref{alg: lsc prox conjugate} with stepsize $\gamma>0$. Let $\phi^{*}=\left(a^{*},u^{*}\right)$ be the minimizer of problem \eqref{eq: min conjugate g}. If $\left(a_{n}\right)_{n\in\mathbb{N}}$ is chosen such that 
$
{\gamma+2a^{*}} \geq 2a_{n}
$
for all $n\in\mathbb{N}$. Then $\left(\phi_{n}\right)_{n\in\mathbb{N}}$ is bounded. If $\Phi_{lsc}^\mathbb{R}$ is finite dimensional and $1 -(a_{n}-a_{n+1})\gamma \geq 0$ diverges from zero then there exists a convergence subsequence $\left(u_{n_{k}}\right)_{k\in\mathbb{N}}$ which converges to a minimizer.
\end{theorem}
\begin{proof}
Let $\phi=\phi^{*}$ in \eqref{eq: PPA conjugate 1st estimate} be the minimizer of problem \eqref{eq: min conjugate g}, we have
\begin{align*}
0\geq h\left(\phi^{*}\right)-h\left(\phi_{n+1}\right) & \geq\frac{1}{2\gamma}\left[\left\Vert u^{*}-u_{n+1}\right\Vert ^{2}-\left\Vert u^{*}-u_{n}\right\Vert ^{2}\right]+\left(1-\frac{2\left(a^{*}-a_{n+1}\right)}{\gamma}\right)\frac{\left\Vert u_{n}-u_{n+1}\right\Vert ^{2}}{2\gamma}.
\end{align*}
If $\gamma+2a^{*}\geq 2a_{n}$ for all $n\in\mathbb{N}$ then we have $\left\Vert u^{*}-u_{n}\right\Vert ^{2}\geq\left\Vert u^{*}-u_{n+1}\right\Vert ^{2}.$
This infers that $\left(u_{n}\right)_{n\in\mathbb{N}}$ is bounded. From the assumption of $(a_n)_{n\in\mathbb{N}}$, we conclude that $(\phi_n)_{n\in\mathbb{N}}$ is bounded.
For the second assertion, we assume that $\Phi_{lsc}^\mathbb{R}$ is finite dimensional. Since $(\phi_n)_{n\in\mathbb{N}}$ is bounded, there exists a convergence subsequence $\left(\phi_{n_{k}+1}\right)_{k\in\mathbb{N}}$ which converges to some $\bar{\phi} = (\bar{a},\bar{u}) \in \Phi_{lsc}^\mathbb{R}$. 

We claim that if $u_{n}-u_{n+1}\to0$
then $\bar{u}$ belongs to the minimizer. Indeed, if it true, we have that $u_{n_{k}}\to\bar{u}$
and from \eqref{eq: PPA conjugate 1st estimate}, we have
\begin{align}
    & h(\phi)-h(\bar{\phi}) \geq h(\phi) - \liminf_{k\to \infty} h(\phi_{n_k}+1) \nonumber\\
    & \geq\liminf_{n\to \infty} \frac{1}{2\gamma}\left[\left\Vert u-u_{n_k+1}\right\Vert ^{2}-\left\Vert u-u_{n_k}\right\Vert ^{2}\right]+\left(1-\frac{2\left(a-a_{n_k+1}\right)}{\gamma}\right)\frac{\left\Vert u_{n_k}-u_{n_k+1}\right\Vert ^{2}}{2\gamma} =0,
\end{align}
for all $\phi \in \Phi_{lsc}^\mathbb{R}$. Hence $\bar{\phi}$ is a minimizer of $h$.
Now we just need to prove that $u_n-u_{n+1}\to 0$. From \eqref{eq: PPA conjugate 1st estimate} with $\phi=\phi_n$, for any $n\in\mathbb{N}$
\begin{equation*}
    h\left(\phi_{n}\right)-h\left(\phi_{n+1}\right) \geq\left(1-\frac{\left(a_{n}-a_{n+1}\right)}{\gamma}\right)\frac{\left\Vert u_{n}-u_{n+1}\right\Vert ^{2}}{\gamma}.
\end{equation*}
As $h$ is bounded from below, we have that 
\[
\sum_{n=0}^{+\infty} \left(1-\frac{\left(a_{n}-a_{n+1}\right)}{\gamma}\right)\frac{\left\Vert u_{n}-u_{n+1}\right\Vert ^{2}}{\gamma} <+\infty,
\]
which implies either $u_n-u_{n+1}\to 0$ or $a_n-a_{n+1} \to \gamma$. We obtain the former thanks to the assumption on $(a_n)_{n\in\mathbb{N}}$. This finishes the proof.
\end{proof}

\begin{remark}
From the above theorem, we only require that $a_n-a_{n+1} \leq \gamma$ and do not converge to $\gamma$, so that $u_n-u_{n+1}\to 0$. This is crucial to obtain convergence to the minimizer. On the other hand, we need information on the solution in order to choose $\gamma$ or $a_n$ such that $\gamma+2a^* \geq 2a_n$. This is one of the weakness of Algorithm \ref{alg: lsc prox conjugate}. 
\end{remark}
To avoid using information on the minimizer, we impose a restriction on the $\Phi_{lsc}^\mathbb{R}$-subgradient of the update \eqref{eq: proximal update conjugate}. Hence, we propose Algorithm \ref{alg: lsc prox conjugate v1}.
\begin{algorithm}[H]\caption{$\lsc^\R$-Proximal Algorithm on $\Phi_{lsc}^\mathbb{R}$ Algorithm}\label{alg: lsc prox conjugate v1}
\textbf{Initialize:} Choose $\gamma >0, \phi_0=(a_0,u_0)\in \Phi_{lsc}^\mathbb{R}$\\
{\normalfont\textbf{Update: }}{
For $n\in\N$,
\begin{itemize}
    \item Find $\phi_{n+1}=(a_{n+1},u_{n+1})$ such that 
    \begin{equation}
    \label{eq: proximal update conjugate v1}
    \frac{u_n -u_{n+1}}{\gamma} \in \partial_{\Phi}^\mathbb{R} h(\phi_{n+1}) \textbf{ and } \left\Vert u_{n}-u_{n+1}\right\Vert ^{2}\leq a_{n+1}-a_{n}.
    \end{equation}
    \item \textbf{If} there is no such $\phi_{n+1}$, \textbf{return} $\phi_n$.
\end{itemize}}
\end{algorithm}
\vspace{0.5cm}
Update \eqref{eq: proximal update conjugate v1} implies that $a_{n+1}\geq a_n$, so one needs to take a non-decreasing sequence $(a_n)_{n\in\mathbb{N}}$.

We state the convergence results of Algorithm \ref{alg: lsc prox conjugate v1} in the following theorem.
\begin{theorem}
Let $h:\Phi_{lsc}^\mathbb{R} \to (-\infty,+\infty]$ be a proper lsc function on $\Phi_{lsc}^\mathbb{R}$ which is bounded from below. Let $(\phi_{n})_{n\in\mathbb{N}}$ be the sequence generated by Algorithm \ref{alg: lsc prox conjugate} with stepsize $\gamma>0$.
For any $\phi\in \Phi_{lsc}^\mathbb{R}$, we have
\begin{align}
    h\left(\phi\right)-h\left(\phi_{n+1}\right) 
    &\geq\frac{1}{2\gamma^{2}}\left[\left(a-a_{n+1}\right)^{2}-\left(a-a_{n}\right)^{2}\right]+\frac{1}{2\gamma}\left[\left\Vert u-u_{n+1}\right\Vert ^{2}-\left\Vert u-u_{n}\right\Vert ^{2}\right] \nonumber\\
    & +\frac{1}{2\gamma}\left\Vert u_{n}-u_{n+1}\right\Vert ^{2}\left[\frac{\left\Vert u_{n}-u_{n+1}\right\Vert ^{2}}{\gamma}+1\right].
    \label{eq: prox conj v1 1st inequality}
\end{align}
In addition, 
$h(\phi_n)$ is non-increasing and $(\phi_n)_{n\in\mathbb{N}}$ is bounded. When $\Phi_{lsc}^\mathbb{R}$ is finite dimensional, every convergent subsequence of $(\phi_n)_{n\in\mathbb{N}}$ converges to a minimizer of $h$.
\end{theorem}
\begin{proof}
Let us start from update \eqref{eq: proximal update conjugate v1}, for any $\phi\in \Phi_{lsc}^\mathbb{R}$, 
\begin{equation}
\label{eq: lsc prox conj 1st update}
    h\left(\phi\right)-h\left(\phi_{n+1}\right) \geq-\frac{a-a_{n+1}}{\gamma^{2}}\left\Vert u_{n}-u_{n+1}\right\Vert ^{2}+\frac{1}{\gamma}\left\langle u_{n}-u_{n+1},u-u_{n+1}\right\rangle.
\end{equation}
The restriction of $\Vert u_n-u_{n+1}\Vert$, gives us
\begin{align}
h\left(\phi\right)-h\left(\phi_{n+1}\right) & \geq \frac{1}{2\gamma^{2}}\left[\left\Vert u_{n}-u_{n+1}\right\Vert ^{4}+\left(a-a_{n+1}\right)^{2}-\left(a-a_{n+1}+\left\Vert u_{n}-u_{n+1}\right\Vert ^{2}\right)^{2}\right] \nonumber\\
 & +\frac{1}{2\gamma}\left[\left\Vert u-u_{n+1}\right\Vert ^{2}-\left\Vert u-u_{n}\right\Vert ^{2}+\left\Vert u_{n}-u_{n+1}\right\Vert ^{2}\right] \nonumber\\
 & \geq\frac{1}{2\gamma^{2}}\left[\left(a-a_{n+1}\right)^{2}-\left(a-a_{n}\right)^{2}\right]+\frac{1}{2\gamma}\left[\left\Vert u-u_{n+1}\right\Vert ^{2}-\left\Vert u-u_{n}\right\Vert ^{2}\right] \nonumber\\
 & +\frac{1}{2\gamma}\left\Vert u_{n}-u_{n+1}\right\Vert ^{2}\left[\frac{\left\Vert u_{n}-u_{n+1}\right\Vert ^{2}}{\gamma}+1\right],\nonumber
\end{align}
which is \eqref{eq: prox conj v1 1st inequality}. Let $\phi=\phi_n$ in \eqref{eq: prox conj v1 1st inequality}, we obtain
\begin{align}
h\left(\phi_n\right)-h\left(\phi_{n+1}\right) 
 & \geq\frac{(a_n-a_{n+1})^2}{2\gamma^{2}} + \frac{1}{2\gamma}\left\Vert u_{n}-u_{n+1}\right\Vert ^{2}\left[\frac{\left\Vert u_{n}-u_{n+1}\right\Vert ^{2}}{\gamma}+2\right] \geq 0,\nonumber.
\end{align}
This implies that $h(\phi_n)$ is non-increasing for all $n\in\mathbb{N}$. Moreover, since $h$ is bounded from below, the above inequality means that 
\[
\sum_{i=0}^{+\infty} \frac{(a_i-a_{i+1})^2}{2\gamma^{2}} + \frac{1}{2\gamma}\left\Vert u_{i}-u_{i+1}\right\Vert ^{2}\left[\frac{\left\Vert u_{i}-u_{i+1}\right\Vert ^{2}}{\gamma}+2\right] <+\infty,
\]
so $u_n-u_{n+1} \to 0$ and $a_n-a_{n+1}\to 0$. The boundedness and subsequence convergence of $(\phi_n)_{n\in\mathbb{N}}$ to the minimizer follows in the same manner as in Theorem \ref{thm: prox conj convergence}.
\end{proof}

\section{Examples setup}
\label{apendix: numerical 1}
\paragraph{Function $f$:}
For examples \ref{ex1} and \ref{ex2: nonconvex example}, we require the $\Phi_{lsc}^\mathbb{R}$-subdifferentials of $f(x)=x^4$. Hence, for every $x_0\in\mathbb{R}$, we need to find $\phi_0 (x) = -a_0 x^2 +u_0 x$ such that $
x_0 \in \argmin_{x\in\mathbb{R}}\ f(x)-\phi_0(x)$. Since $f-\phi_0$ is smooth and it is convex if $a\geq 0$, we can take advantage of the gradient and its second derivative to obtain global minimizer. 
In fact, we found out that we can take a negative value $a\geq -2x_0$ so that $x_0$ is still the global minimizer. The $\Phi_{lsc}^\mathbb{R}$-subdifferentials of $f$ has the form,
\begin{equation}
\label{ex1: phi subgrad f}
\partial_{lsc}^\mathbb{R} f(x_0) :=\{ \phi=(a,u)\in \Phi_{lsc}^{\mathbb{R}} \ : a\geq-2x_{0}^{2},b=4x_{0}^{3}+2ax_{0}\}.
\end{equation}
To implement $\Phi_{lsc}^\mathbb{R}$-proximal operator on $f$ with parameter $\sigma > 0$, we need to specify the sequence $(\bar{a}_n)_{n\in\mathbb{N}}, \bar{a}_n \geq -1/2\sigma$ for all $n\in\mathbb{N}$. Then the proximal update for Algorithm \ref{alg: lsc Chambolle-Pock} on the function $f + \langle y_{n+1},\cdot \rangle$ reads as
\begin{align*}
\text{Solve for } x_{n+1}: \ & 4x^{3}+\left(2\bar{a}_{n}+\frac{1}{\sigma}\right)x=\left(\frac{1}{\sigma}+2\bar{a}_{n}\right)x_{n} -y_{n+1}\\
\bar{a}_{n+1} & = \bar{a}_n +2x_{n+1}^2 -\varepsilon,
\end{align*}
where $\varepsilon=10^{-3}$ is a damping term added to $\bar{a}_{n+1}$.
While for Algorithm \ref{alg: lsc Chambolle-Pock full L=Id} and \ref{alg: lsc Chambolle-Pock full L=Id v2}, it is $\phi_{n+1}=(a_{n+1},u_{n+1})$ instead of $y_{n+1}$,
\begin{align}
\text{Solve for } x_{n+1}: \ & 4x^{3}+\left(2\bar{a}_{n}-2a_{n+1}+\frac{1}{\sigma}\right)x=\left(\frac{1}{\sigma}+2\bar{a}_{n}\right)x_{n} -y_{n+1}\nonumber\\
\bar{a}_{n+1} & = \bar{a}_n +2x_{n+1}^2 - a_{n+1} -\varepsilon,
\label{ex1: primal update}
\end{align}

The cubic equation always has a real solution. In case of multiple solutions, we choose the one that is the closest to the previous iteration $x_n$. To find solution to the cubic equation, we employ the package NUMPY in PYTHON.

\paragraph{Function $g$:} 
For function $g(x)=x^2$, this is a well-known convex function, with its convex conjugate $g^*(y) = y^2/4$. For completeness, the proximal update $y_{n+1}\in \text{prox }_{\tau g^*}(y_n+\tau \bar{x}_n)$ is 
\[ y_{n+1}\in\arg\min_{y\in \mathbb{R}} g^*(y) -  \bar{x}_n y +\frac{1}{2\tau} (y-y_n)^2 \Leftrightarrow y_{n+1} = \frac{2 (y_n+\tau \bar{x}_n)}{\tau+2}.
\]
This will be used in Algorithm \ref{alg: lsc Chambolle-Pock} and standard Chambolle-Pock algorithm.

To compute $\Phi_{lsc}^\mathbb{R}$-conjugate and $\Phi_{lsc}^\mathbb{R}$-subdifferentials, we focus on the generalized function $h(x)=c x^2$, with $c\in\mathbb{R}$. The $\Phi_{lsc}^\mathbb{R}$-conjugate, $h^*_\Phi (\phi) = \sup_{x\in\mathbb{R}} \phi(x)-h(x)$, is 
\[
h^*_\Phi (\phi) = h^*_\Phi (a,u) =
\begin{cases}
0 & a=-c,u=0\\
\frac{u^{2}}{4\left(a+c\right)} & a>-c\\
+\infty & \text{otherwise}.
\end{cases}
\]
By the definition, $h^*_\Phi$ is a convex function, so we can calculate its convex subgradient at $\phi_0=(a_0,u_0)\in\Phi_{lsc}^\mathbb{R}$ which is 
\begin{align}
\partial h^*_\Phi (\phi_0) & =\{ \varphi=(\bar{a},\bar{u}) \in \Phi_{lsc}^\mathbb{R}:(\forall \phi\in\Phi_{lsc}^\mathbb{R})\ h^*_\Phi (\phi)-h^*_\Phi (\phi_0) \geq \langle \varphi , \phi-\phi_0\rangle_\Phi\} \nonumber\\
& = \begin{cases}
\bar{u}=\frac{u_{0}}{2\left(a_{0}+c\right)},\bar{a}=-\bar{u}^{2} & \text{If } a_0\geq -c \\
\emptyset & \text{otherwise}.
\end{cases}
\label{ex1: g* convex subgrad}
\end{align}
Hence, in our example, we can take $c=1$ and use the classical proximal operator in the dual update of Algorithm \ref{alg: lsc Chambolle-Pock full L=Id},
\begin{align*}
    u_{n+1} & =\frac{2\left(a_{n+1}+1\right) \left(u_{n}- \tau(2x_{n}-x_{n-1}) \right)}{\tau+2(a_{n+1}+1)}, \\
    \frac{a_{n}-a_{n+1}}{\tau}+2x_{n}^{2}-x_{n-1}^{2} & = -\left[\frac{\left(u_{n}- \tau(2x_{n}-x_{n-1}) \right)}{\tau+2(a_{n+1}+1)} \right]^{2},\\
    a_{n+1} &\geq -1. 
\end{align*}
To obtain $a_{n+1}$, we have to solve the middle equation which can be written as 
\[
4a_{n+1}^{3}+4a_{n+1}^{2}\left(\tau+2-A_{1}\tau\right)+a_{n+1}\left(\left(\tau+2\right)^{2}-4\left(\tau+2\right)\tau A_{1}\right)-\left(\tau+2\right)^{2}\tau A_{1}-A_{2}\tau =0,
\]
where $A_{1}=\frac{a_{n}}{\tau}+2x_{n}^{2}-x_{n-1}^{2},A_{2}=\left(\frac{u_{n}}{\tau}-2x_{n}+x_{n-1}\right)^{2}\tau^{2}$.
This equation always has a real solution and we will choose the one that is the closest to the previous iteration $a_n$. The condition $a_{n+1}\geq -1$ needs to be satisfied for all $n\in\mathbb{N}$. \textbf{Therefore, we add $a_{n} <-1$ to our stopping criterion beside $\bar{a}_n <-1/2\sigma$ in Algorithm \ref{alg: lsc Chambolle-Pock full L=Id}.}

In the same manner, $\Phi_{lsc}^\mathbb{R}$-subdifferentials of $h^*_\Phi$ at $\phi_0=(a_0,u_0)\in\Phi_{lsc}^\mathbb{R}$ has the form,
\begin{align}
\partial_{lsc}^\mathbb{R} h^*_\Phi (\phi_0) & =\{ w_0 \in \mathbb{R}:(\forall \phi\in\Phi_{lsc}^\mathbb{R})\ h^*_\Phi (\phi)-h^*_\Phi (\phi_0) \geq \phi(w_0)-\phi_0 (w_0)\} \nonumber\\
& =\begin{cases}
\mathbb{R} & \text{If } \phi_0=\left(-c,0\right)\\
\frac{u_{0}}{2\left(a_{0}+c\right)} & \text{If } a_{0}>-c\\
\emptyset & \text{otherwise}.
\end{cases}
\label{ex1: g* lsc subgrad}
\end{align}
Since sum rule does not work for $\Phi_{lsc}^\mathbb{R}$-subdifferentials on $\Phi_{lsc}^\mathbb{R}$, to implement Algorithm \ref{alg: lsc Chambolle-Pock full L=Id v2}, we take a closer look at the dual update. We want to find $(x^*,\phi^*)\in \mathbb{R}\times\Phi_{lsc}^\mathbb{R}$ such that $x^* \in \partial_{lsc}^\mathbb{R} g^*_\Phi (\phi^*)$ or to minimize
\begin{equation}
\label{ex1: dual lsc lsc subgrad problem}
\min_{\phi\in\Phi_{lsc}^\mathbb{R}} g^*_\Phi (\phi) -\phi(x^*).
\end{equation}
To apply Algorithm \ref{alg: lsc Chambolle-Pock full L=Id v2}, we use $x^*=\bar{x}_n = 2x_n-x_{n-1}$. However, notice that $\phi(x)$ is nonlinear, we have to split $\phi(\bar{x}_n)$, so we have to solve for $\phi_{n+1}$ the problem,
\[
\min_{\phi\in\Phi_{lsc}^\mathbb{R}} g^*_\Phi (\phi) -(2\phi(x_n)-\phi(x_{n-1})) = \min_{\phi\in\Phi_{lsc}^\mathbb{R}} g^*_\Phi (\phi) - K(\phi),
\]
where $K(\phi) = 2\phi(x_n)-\phi(x_{n-1})$.
The above problem can be written in term of $\Phi_{lsc}^\mathbb{R}$-subdifferentials as
\[
0\in \partial_{lsc}^\mathbb{R} (g^*_\Phi -K)(\phi_{n+1}) \Leftrightarrow 2x_n -x_{n-1} \in  \partial_{lsc}^\mathbb{R} (g^*_\Phi - k)(\phi_{n+1}),
\]
with $k(\phi) = a_\phi\Vert x_n-x_{n-1}\Vert^2$. The $\Phi_{lsc}^\mathbb{R}$-subdifferentials of $(g^*_\Phi-k)(\phi_0)$ is
\[
\partial_{lsc}^\mathbb{R} (g^*_\Phi-k)(\phi_0) = \{x\in \mathbb{R}: (u_0-2(a_0+c)x)^2 = 8(a_0+c)^2  (x_n-x_{n-1})^2\}.
\]
We would like to use Algorithm \ref{alg: lsc prox conjugate} in the dual update of Algorithm \ref{alg: lsc Chambolle-Pock full L=Id v2}, so we want to find the next iterate $\phi_{n+1}=(a_{n+1},u_{n+1})$ such that 
\[
\frac{u_n-u_{n+1}}{\tau} +2x_n -x_{n-1} \in \partial_{lsc}^\mathbb{R} (g^*_\Phi-k)(\phi_{n+1}),
\]
this gives us 
\[
u_{n+1} = \frac{2\left(a_{n+1}+c\right)\left(u_{n}-\tau\sqrt{2}\left(x_{n}-x_{n-1}\right)+\tau\left(2x_{n}-x_{n-1}\right)\right)}{2\left(a_{n+1}+c\right)+\tau},
\]
and we slowly decrease $a_{n+1}=a_n-\varepsilon$ with $\varepsilon=10^{-3}$, instead of restricting $(a_n)_{n\in\mathbb{R}}$ as in Algorithm \ref{alg: lsc prox conjugate v1}. Again, we require $a_n \geq -1$ for all $n\in\mathbb{N}$ and we will put this as a stopping criterion in Algorithm \ref{alg: lsc Chambolle-Pock full L=Id v2}.

\begin{remark}
    To use Algorithm \ref{alg: lsc prox conjugate v1}, we need to choose $(a_n)_{n\in\mathbb{N}}$ increasing such that $(u_n-u_{n+1})^2 \leq a_{n+1}-a_n$. One needs to start with a large step-length $a_{n+1}-a_n$ and then decrease it. However, a large value $a_{n+1}$ can affect the primal update which can make it fail to find $x_{n+1}$ or stopping the algorithm (see \eqref{ex1: primal update}).
    Moreover, to find $a_{n+1}$, we need to solve the inequality $(u_n-u_{n+1})^2 \leq a_{n+1}-a_n$. Using the update $u_{n+1}$ above, we have
    \[
    \left[ u_n-\frac{2\left(a_{n+1}+c\right)\left(u_{n}-\tau\sqrt{2}\left(x_{n}-x_{n-1}\right)+\tau\left(2x_{n}-x_{n-1}\right)\right)}{2\left(a_{n+1}+c\right)+\tau} \right]^2 \leq a_{n+1}-a_n.
    \]
    This is a cubic inequality in $a_{n+1}$ which will take a lot of efforts to solve.
    That is why we refrain from using Algorithm \ref{alg: lsc prox conjugate v1} in Algorithm \ref{alg: lsc Chambolle-Pock full L=Id v2}.
    On the other hand, this specific example require $a_n\geq -1$, while $\phi^*=(a^*,u^*)=(-1,0)$ is the dual solution. It does not make sense to choose an increasing sequence $(a_n)_{n\in\mathbb{N}}$. 
\end{remark}
\end{appendices}
\end{document}